\documentclass[11pt,letterpaper]{amsart}
\usepackage{amsmath, amsthm, amssymb}
\usepackage{parskip}
\newtheorem{proposition}{Proposition}
\newtheorem{thm}{Theorem}[section]
\newtheorem{lem}[thm]{Lemma}

\newtheorem{mydef}{Definition}[section]

\newtheorem{rem}{Remark}[section]
\newcommand{\bFormula}[1]{
\begin{equation} \label{#1}}
\newcommand{\eF}{\end{equation}}

\newcommand{\vu}{{\bf u}}
\newcommand{\vx}{{\bf x}}

\newcommand{\vphi}{{\boldsymbol \varphi}}
 \def\R{\mathbb{R}}

  \def\N{\mathbb{N}}

\newcommand\bb[1]{\mathbf{#1}}

\newtheorem*{lem*}{Lemma}

\newcommand{\eqnb}{\begin{equation}}
\newcommand{\eqne}{\end{equation}}
\title{Global weak solutions in nonlinear 3D thermoelasticity}
\author{Tomasz Cie\'{s}lak}
\address{Institute of Mathematics, Polish Academy of Sciences}
\email{cieslak@impan.pl}
\author{Boris Muha}
\address{Department of Mathematics, Faculty of Science, University of Zagreb}
\email{borism@math.hr}
\author{Sr\dj{}an Trifunovi\'{c}}
\address{Department for Mathematics and Informatics, University of Novi Sad}
\email{srdjan.trifunovic@dmi.uns.ac.rs}
\begin{document}
\sloppy

\maketitle
\thispagestyle{empty}

\begin{abstract}
Here we study a nonlinear thermoelasticity hyperbolic-parabolic system describing the balance of momentum and internal energy of a heat-conducting elastic body, preserving the positivity of temperature. So far, no global existence results in such a natural case were available. Our result is obtained by using thermodynamically justified variables which allow us to obtain an equivalent system in which the internal energy balance is replaced with entropy balance. For this system, a concept of weak solution with defect measure is introduced, which satisfies entropy inequality instead of balance and has a positive temperature almost everywhere. Then, the global existence, consistency and weak-strong uniqueness are shown in the cases where heat capacity and heat conductivity are both either constant or non-constant. Let us point out that this is the first result concerning global existence for large initial data in nonlinear thermoelasticity where the model is in full accordance with the laws of thermodynamics.
\end{abstract}

\section{Introduction and main result}\label{Intro}

Let $\Omega\in \R^3$ be a bounded Lipschitz domain. We study the physical model of a heat-conducting elastic body described by two state variables -- the displacement $\textbf{u}:[0,T)\times \Omega\rightarrow \R^3$ and temperature $\theta:[0,T)\times \Omega\rightarrow \R$. The two governing equations are the balance of momentum and internal energy. In this paper, we will focus on the following two models, physically viable for high and low temperatures, respectively. Their derivation is given in section \ref{Phys}.

\subsection{The model with constant heat capacity and heat conductivity}
We deal with the following problem:
\begin{equation}\label{system0}
\begin{cases}
\mathbf{u}_{tt} = \nabla \cdot ( \nabla\mathbf{u}   -\mu \theta I),& \quad \text{ in } (0,T)\times \Omega,\\
\theta_t - \Delta \theta =- \mu \theta \nabla\cdot \mathbf{u}_t,& \quad \text{ in }  (0,T)\times\Omega,\\
\mathbf{u} = 0, \partial_n\theta=0,& \quad  \mbox{on } (0,T)\times\partial \Omega,\\
\mathbf{u}(0,\cdot)=\mathbf{u}_0,~ \mathbf{u}_t(0,\cdot) = \mathbf{v}_0,~ \theta(0,\cdot)=\theta_0>0,&
\end{cases}
\end{equation}
where constant $\mu>0$ is a given quantity.

Such a system arises as a simplest nonlinear thermoelasticity model. So far, the problem of global existence of solutions has not been solved for this system.

First of all, let us notice that the equation $\eqref{system0}_2$ can formally be divided by $\theta>0$ to obtain the entropy balance
\begin{equation}
\left(\ln \theta+\mu\nabla\cdot \mathbf{u}\right)_t-\Delta \ln \theta=|\nabla \ln \theta|^2\;\;\mbox{on}\;\; (0,T)\times\Omega. \label{ent:eq1}
\end{equation}
Next, multiplying $\eqref{system0}_1$ with $\mathbf{u}_t$, summing up with $\eqref{system0}_2$ and integrating over $(0,t)\times\Omega$ for $t\in (0,T]$ gives the total energy balance
\begin{eqnarray*}
    &&\int_\Omega \theta(t)+ \frac12 \int_\Omega |\mathbf{u}_t|^2(t) + \frac12 \int_\Omega |\nabla\mathbf{u} |^2(t)\nonumber \\
   && = \int_\Omega \theta_0+ \frac12 \int_\Omega |\mathbf{v}_0|^2 + \frac12 \int_\Omega |\nabla\mathbf{u}_0 |^2,
\end{eqnarray*}
while subtracting $\eqref{ent:eq1}$ integrated over $(0,t)\times\Omega$ from the total energy balance gives rise to the total dissipation balance
\begin{eqnarray}\label{total:dissipation:balance:lin}
     &&\int_\Omega (\theta-\ln\theta)(t) +  \frac12 \int_\Omega |\mathbf{u}_t|^2(t) + \frac12 \int_\Omega |\nabla\mathbf{u} |^2(t)+ \int_0^t\int_\Omega |\nabla {\ln\theta} |^2 \nonumber \\
     &&= \int_\Omega (\theta_0 - \ln\theta_0)+\frac12 \int_\Omega |\mathbf{v}_0|^2+ \frac12 \int_\Omega |\nabla \mathbf{u}_0|^2.
\end{eqnarray}
Hence, denoting $\tau:=\ln\theta$, we introduce the following definition of weak solution:
\noindent
\begin{mydef}\label{weaksolution}
We say that $(\mathbf{u},\tau)$ is a weak solution to the problem $(\ref{system0})$ with defect measure if:
\begin{itemize}
    \item The initial data is of regularity
    \begin{eqnarray*}
    \mathbf{u}_0\in H_0^1(\Omega),\quad \mathbf{v}_0\in L^2(\Omega),\quad
    \tau_0,e^{\tau_0} \in L^1(\Omega).
\end{eqnarray*}
    Moreover\footnote{In Appendix, it is proved that $\bb{u}_t \in C_w(0,T; L^2(\Omega))$ for any weak solution.}, $\mathbf{v}_0$ is attained in $C_w(0,T; L^2(\Omega))$ and $\lim\limits_{t\to 0}\int_\Omega{\tau}(t)\geq \int_{\Omega}{\tau}_0$;
    \item $\mathbf{u} \in L^\infty(0,T; H_0^1(\Omega))\cap W^{1,\infty}(0,T;L^2(\Omega))$,\\
    $\tau\in L^\infty(0,T; L^1(\Omega))\cap L^2(0,T; H^1(\Omega))$,\\
    $e^\tau \in L^\infty(0,T;L^1(\Omega))$;
    \item The momentum equation
    \begin{eqnarray}
        &&\int_0^T\int_\Omega \mathbf{u}_t \cdot \boldsymbol\varphi_t - \int_0^T\int_\Omega \nabla\mathbf{u} : \nabla \boldsymbol\varphi + \mu\int_0^T\left\langle \theta, \nabla\cdot\boldsymbol\varphi \right\rangle_{[\mathcal{M}^+;C](\overline{\Omega})}\nonumber\\
        &&=-\int_\Omega \mathbf{v}_0\cdot \boldsymbol\varphi, \label{momweak}
    \end{eqnarray}
holds for all test function $\boldsymbol\varphi \in C_0^\infty([0,T)\times \Omega)$, where $\theta \in L^\infty(0,T;\mathcal{M}^+(\overline{\Omega}))$ satisfies
\begin{eqnarray*}
d\theta = e^\tau dx+f,
\end{eqnarray*}
and $f\geq 0$ is a singular part of $\theta$ supported on a set of measure zero;
\item The entropy equation
\begin{eqnarray}\label{entweak}
   &&\int_0^T \int_\Omega \tau \phi_t-\int_0^T \int_\Omega \nabla \tau \cdot \nabla \phi + \mu\int_0^T \int_\Omega \mathbf{u}_t \cdot \nabla \phi + \left\langle\boldsymbol\sigma, \phi \right\rangle_{[\mathcal{M}^+;C]([0,T]\times \overline{\Omega})} \nonumber \\
   &&= -\int_\Omega \tau_0 \phi_0, 
\end{eqnarray}
holds for all $\phi \in C_0^\infty([0,T)\times \Omega)$, where $\boldsymbol\sigma \in \mathcal{M}^+([0,T]\times \overline{\Omega})$ satisfies
\begin{eqnarray*}
    \boldsymbol\sigma\geq |\nabla\tau|^2;
\end{eqnarray*}
\item The energy inequality
\begin{eqnarray}\label{enweak}
   &&\int_\Omega d\theta(t)+ \frac12 \int_\Omega |\mathbf{u}_t|^2(t) + \frac12 \int_\Omega |\nabla\mathbf{u} |^2(t) \nonumber \\
   &&\leq \int_\Omega e^{\tau_0}+ \frac12 \int_\Omega |\mathbf{v}_0|^2 + \frac12 \int_\Omega |\nabla\mathbf{u}_0 |^2,
\end{eqnarray}
holds for all $t\in (0,T]$.
\end{itemize}
\end{mydef}

\begin{rem}\label{RemarkWS}
(1) In this definition, the internal energy equation $\eqref{system0}_2$ is replaced with the entropy equation $\eqref{ent:eq1}$. Let us point out the key ideas behind this. First, the nonlinearity $- \mu \theta \nabla\cdot \mathbf{u}_t$ in $\eqref{system0}_2$ is not expected to be even integrable for weak solutions, and the compactness for such a term seems unreachable even in 1D (see \cite{CGM}). The entropy equation provides us with some additional information, in particular the total dissipation identity (inequality) $\eqref{total:dissipation:balance:lin}$, which gives us the estimates of $\ln\theta$ and $\nabla \ln \theta$ and thus ensures the positivity of $\theta$. It also plays a crucial role in the proof of weak-strong uniqueness. Finally, the measure $\boldsymbol\sigma$ appears since the nonlinear term $|\nabla \ln \theta|^2$ is only bounded $L_{t,x}^1$ and therefore we cannot rule out concentrations which would give rise to a defect measure. \\
(2) In the above definition, the unknown functions are the displacement $\mathbf{u}$ and the thermal entropy $\tau=\ln\theta$. However, this identity, due to lack of uniform integrability of $\theta$, might not hold on a set of measure zero, which is an issue since $\theta$ is only a measure. Whether or not this integrability can be obtained is an interesting and important problem.\\
(3) Such a concept of a weak solution was inspired by the work of Feireisl and Novotny \cite{FN}, where heat-conducting compressible fluids governed by the full Navier-Stokes-Fourier system are studied. However, the idea of weak solutions with defect measures dates back to work of DiPerna and Lions \cite{DL} and Alexandre and Villani \cite{AV}, where the concept of weak solutions with defect measure was studied in the context of the Boltzmann equation. 
\\
(4) The concept of weak thermodynamical laws, presented in the form of inequalities rather than equalities, is widely discussed in mathematical literature \cite{FN}. Interested readers are also referred to \cite[Section 4]{serrin96}, where the implications of assuming inequalities in the first laws are explored. It is demonstrated that these inequalities imply the Clausius-Duhem inequality and an energy balance equation. In the author's words, this approach ``in many ways implies equivalent or only modestly different results''.
\end{rem}

The main result for this model concerns global-in-time existence of solutions to (\ref{system0}) in the sense of the above definition.
Moreover, we show that solutions satisfying the above definition are consistent, i.e. if a weak solution satisfies certain additional regularity property, then it solves (\ref{system0})
pointwise. Next, we show that in case a classical solution starting from initial data $(\theta_0, \mathbf{u}_0, \mathbf{v}_0)$ exists, then our weak solution and classical solution coincide. To be more precise, the following theorem holds.

\begin{thm}[\textbf{Main result I}]\label{main:lin}
Let $\mathbf{u}_0\in H_0^1(\Omega)$, $\mathbf{v}_0\in L^2(\Omega)$, $\tau_0 \in L^1(\Omega)$ with $e^{\tau_0} \in L^1(\Omega)$. Then, one has the following:
\begin{enumerate}
    \item[] \textbf{Global existence}. There exists a global weak solution $(\mathbf{u},\tau)$ in the sense of Definition \ref{weaksolution};
    \item[] \textbf{Consistency}. If a weak solution in the sense of Definition \ref{weaksolution} is smooth, then the solution solves (\ref{system0}) pointwise;
    \item[] \textbf{Weak-strong uniqueness}. Let $(\mathbf{u},\tau)$ be a weak solution in the sense of Definition \ref{weaksolution}. Then, it satisfies the relative entropy inequality $\eqref{relent}$. Moreover, if $(\tilde{\mathbf{u}},\tilde{\tau})$ is a classical solution to the problem $(\ref{system0})$ and
\begin{eqnarray*}
    \tau_0=\tilde{\tau}_0, \qquad \mathbf{u}_0 = \tilde{\mathbf{u}}_0,\qquad \mathbf{v}_0 = \tilde{\mathbf{v}}_0,
\end{eqnarray*}
then
\begin{eqnarray*}
    \mathbf{u}\equiv \tilde{\mathbf{u}}, \qquad \tau\equiv \tilde{\tau}.
\end{eqnarray*}
\end{enumerate}
\end{thm}
The approach developed here was inspired by the theory of Feireisl and Novotny from \cite{FN}, where the authors deal with viscous compressible heat-conducting fluids.


\subsection{The model with non-constant heat capacity and heat conductivity}\label{nonl}
The problem we choose to study is the following:


\begin{eqnarray}\label{system0:nonlin}
    \begin{cases}
   \mathbf{u}_{tt} = \nabla \cdot ( \nabla\mathbf{u}   -\mu \theta I),& \mbox{in }(0,T)\times\Omega,\\
    (\theta+\theta^{\alpha})_t - \nabla\cdot\big((1+\theta^\beta)\nabla\theta\big) =- \mu \theta \nabla\cdot \mathbf{u}_t ,& \mbox{in }(0,T)\times\Omega,\\
     \mathbf{u} = 0, ~\partial_n\theta=0,& \mbox{on } (0,T)\times\partial \Omega, \\
    \mathbf{u}(0,\cdot)=\mathbf{u}_0,~ \mathbf{u}_t(0,\cdot) = \mathbf{v}_0, ~\theta(0,\cdot)=\theta_0,&
    \end{cases}
\end{eqnarray}
where $\alpha>1$ and $\beta,\mu>0$. The entropy equation reads
\begin{eqnarray}
     &&\left(\ln\theta+\frac{\alpha}{\alpha-1}\theta^{\alpha-1}+\mu\nabla\cdot \mathbf{u}\right)_t -\nabla \cdot \left(\frac{(1+\theta^\beta)\nabla\theta}{\theta} \right) \nonumber\\
     &&= \frac{(1+\theta^\beta)|\nabla\theta|^2}{\theta^2},
     \label{system:nonlin}
\end{eqnarray}
on $(0,T)\times\Omega$. The total energy balance for this system is now
\begin{eqnarray}\label{en:id:nonlin}
    &&\int_\Omega( \theta+\theta^\alpha)(t)+ \frac12 \int_\Omega |\mathbf{u}_t|^2(t) + \frac12 \int_\Omega |\nabla\mathbf{u} |^2(t) \nonumber\\
    &&= \int_\Omega (\theta_0+\theta_0^\alpha)+ \frac12 \int_\Omega |\mathbf{v}_0|^2 + \frac12 \int_\Omega |\nabla\mathbf{u}_0 |^2,
\end{eqnarray}
while the total dissipation balance is
\begin{eqnarray}
     &&\int_\Omega \left(\theta+\theta^{\alpha} -\ln\theta-\frac{\alpha}{\alpha-1}\theta^{\alpha-1}\right) +  \frac12 \int_\Omega |\mathbf{u}_t|^2(t) +\frac12 \int_\Omega |\nabla\mathbf{u} |^2(t)\nonumber\\
     &&\quad + \int_0^t\int_\Omega \frac{(1+\theta^\beta)|\nabla\theta|^2}{\theta^2} \nonumber\\
     &&= \int_\Omega \left(\theta_0+\theta_0^{\alpha} -\ln\theta_0-\frac{\alpha}{\alpha-1}\theta_0^{\alpha-1}\right) + \frac12 \int_\Omega |\mathbf{v}_0|^2+ \frac12 \int_\Omega |\nabla\mathbf{u}_0 |^2. \label{diss:nonlin}
\end{eqnarray}
This leads to the following:
\begin{mydef}\label{weaksolution:nonlin}
We say that $(\mathbf{u},\theta)$ is a weak solution to the problem (\ref{system0:nonlin}) with defect measure if:
\begin{itemize}
    \item The initial data is of regularity
    \begin{eqnarray*}
    &&\mathbf{u}_0\in H_0^1(\Omega),\quad \mathbf{v}_0\in L^2(\Omega),
    \quad \theta_0 \in L^\alpha(\Omega),\quad  {\ln\theta}_0 \in L^1(\Omega).
\end{eqnarray*}
    Moreover, $\mathbf{v}_0$ is attained in $C_w(0,T; L^2(\Omega))$  and
    \begin{eqnarray*}
        \lim\limits_{t\to 0}\int_\Omega \left(\ln\theta+\frac{\alpha}{\alpha-1}\theta^{\alpha-1}\right)(t)\geq \int_{\Omega}\left(\ln\theta_0+\frac{\alpha}{\alpha-1}\theta_0^{\alpha-1}\right);
    \end{eqnarray*}
 \item $\mathbf{u}\in L^\infty(0,T; H_0^1(\Omega))\cap W^{1,\infty}(0,T;L^2(\Omega))$,\\
    $\theta \in L^\infty(0,T; L^\alpha(\Omega))$,\\
    $\theta^\beta \in L^1(0,T; L^{3}(\Omega))$, \\
    $\theta^{\frac{\alpha}2+\frac{\beta}2}\in L^{2}(0,T; L^{\frac32}(\Omega))$,\\
    $\nabla \theta^{\frac{s}2} \in L^2((0,T)\times \Omega)$, for all $s\in(0,\beta]$,\\
    $\ln\theta\in L^\infty(0,T; L^1(\Omega))\cap L^2(0,T; H^1(\Omega))$;
    \item The momentum equation
    \begin{eqnarray}\label{momweak:nonlin}
        &&\int_0^T\int_\Omega \mathbf{u}_t \cdot \boldsymbol\varphi_t - \int_0^T\int_\Omega \nabla\mathbf{u} : \nabla \boldsymbol\varphi + \mu\int_0^T\int_\Omega\theta \nabla\cdot\boldsymbol\varphi \nonumber \\
        &&=-\int_\Omega v_0\cdot \boldsymbol\varphi,
\end{eqnarray}
holds for all test function $\boldsymbol\varphi \in C_0^\infty([0,T)\times \Omega)$;
\item The entropy balance
\begin{eqnarray}
   &&\int_0^T \int_\Omega(\ln\theta+\frac{\alpha}{\alpha-1}\theta^{\alpha-1}) \phi_t-\int_0^T \int_\Omega \left(\frac{(1+\theta^\beta)\nabla\theta\cdot \nabla \phi}{\theta} \right) \nonumber \\
   &&\quad + \mu\int_0^T \int_\Omega \mathbf{u}_t \cdot \nabla \phi +  \left\langle \boldsymbol\sigma, \phi \right\rangle_{[\mathcal{M}^+;C]([0,T]\times \overline{\Omega})} \ \nonumber \\
   &&= -\int_\Omega (\ln\theta_0+\frac{\alpha}{\alpha-1}\theta_0^{\alpha-1}) \phi_0, \label{entweak:nonlin}
\end{eqnarray}
holds for all non-negative $\phi \in C_0^\infty([0,T)\times \Omega)$, where $\boldsymbol\sigma \in \mathcal{M}^+([0,T]\times \overline{\Omega})$ satisfies
\begin{eqnarray*}
     \boldsymbol\sigma\geq\frac{(1+\theta^\beta)|\nabla\theta|^2}{\theta^2};
\end{eqnarray*}
\item The energy inequality
\begin{eqnarray}
  &&\int (\theta+\theta^\alpha)(t)+ \frac12 \int_\Omega |\mathbf{u}_t|^2(t) + \frac12 \int_\Omega |\nabla\mathbf{u} |^2(t) \nonumber \\
  &&\leq \int_\Omega (\theta_0+\theta_0^\alpha)+ \frac12 \int_\Omega |\mathbf{v}_0|^2 + \frac12 \int_\Omega |\nabla\mathbf{u}_0 |^2, \label{enweak:nonlin}
\end{eqnarray}
holds for all $t\in (0,T]$.
\end{itemize}
\end{mydef}
We shall show the following result concerning the problem \eqref{system0:nonlin}:
\begin{thm}[\textbf{Main result II}]\label{main:nonlin}
Let $\mathbf{u}_0\in H_0^1(\Omega)$, $\mathbf{v}_0\in L^2(\Omega)$, $\theta_0 \in L^\alpha(\Omega)$ with $\ln\theta_0 \in L^1(\Omega)$, and let $\alpha>1, \beta>0$. Then, one has the following:
\begin{enumerate}
    \item[] \textbf{Global existence}. There exists a weak solution $(\mathbf{u},\theta)$ in the sense of Definition \ref{weaksolution:nonlin};
    \item[] \textbf{Consistency}. If a weak solution in the sense of Definition \ref{weaksolution} is smooth, then the solution solves both (\ref{system0:nonlin}) and (\ref{system:nonlin}) pointwise;
    \item[] \textbf{Weak-strong uniqueness}.  Let $(\mathbf{u},\theta)$ be a weak solution in the sense of Definition $\ref{weaksolution:nonlin}$. Then, it satisfies the relative entropy inequality $\eqref{rel:ent:nonlin}$. Moreover, if $(\tilde{\mathbf{u}},\tilde{\theta})$ is a classical solution to the problem $(\ref{system:nonlin})$ and $\alpha\geq 2, \beta = 2$, then
\begin{eqnarray*}
    \theta_0=\tilde{\theta}_0>0, \qquad \mathbf{u}_0 = \tilde{\mathbf{u}}_0,\qquad \mathbf{v}_0 = \tilde{\mathbf{v}}_0,
\end{eqnarray*}
implies
\begin{eqnarray*}
    \mathbf{u}\equiv \tilde{\mathbf{u}}, \qquad \theta\equiv \tilde{\theta}.
\end{eqnarray*}
\end{enumerate}
\end{thm}

\begin{rem}
The above result also holds when heat conductivity is constant, i.e. when $\alpha>1$ and $\beta=0$.
\end{rem}

\section{Comparison with known results and discussion}\label{compar}

In this section, we compare our result with the known literature. The study of existence of solutions to systems of thermoelasticity dates back to 1960s, however the initial work was related to the linear simplification  (linear thermoelasticity is addressed in detail shortly after). Nonlinear systems of elasticity require more sophisticated methods. Under some particular choices of nonlinear couplings $b(\theta, \mathbf{u}_{tx})$ replacing the term $- \mu \theta \nabla\cdot \mathbf{u}_t$ in the second equation of (\ref{system0}), smooth, classical solutions were known in the literature in the spatially one-dimensional setting locally in time for any data and globally for small data. This was investigated by various authors
and one of the pioneer works was Slemrod's paper \cite{slemrod}. Still, assumptions required to guarantee existence of smooth solutions with nonnegative temperature were so specific, that the simplest system (\ref{system0}) was not covered. See the discussion in \cite{CGM} and results in \cite{CD, jiang, racke}. The classical smooth local-in-time solutions with nonnegative temperature were obtained recently in \cite{CGM}. Let us also recall earlier similar results in \cite{jr, rashi}. Still, even in 1D case, the problem of existence of global-in-time solutions is widely opened. In 1D, one could mention the global existence of measure-valued solutions in \cite[Theorem 2]{CGM}. The regularity of such a solution is very weak and there seems no hope in that case for any sort of weak-strong uniqueness.

Recently, a notable attention is being given to problems in adiabatic thermoelasticity -- the system in which the heat flux is neglected, a valid approximation in thermodynamical systems in which processes happen in a small time interval and thus not allow for heat flux to take effect. From a mathematical point of view, this is a purely hyperbolic system. Adiabatic thermoviscoelastic system was studied by C. Christoforou and A. Tzavaras in \cite{chrtzav}. They use the method of relative entropy in order to prove the zero viscosity limit, thus going from thermoviscoelasticity with heat flux to adiabatic thermoelastic system. Note that the existence of viscous solutions is not shown. Their next results obtained together with Galanopoulou in \cite{chrmytz, chrmytz2} concern some adiabatic thermoelastic problems. In the first paper, a measure-valued solution to the thermoelastic problem is obtained. It requires the use of Young measures as well as an embedding of the considered system into the hyperbolic conservation laws framework. It requires some structural assumptions (see \cite[Remark 1, p. 6182]{chrmytz}). It seems that the solution obtained in \cite{chrmytz2}, again a measure-valued one, requires similar strong structural assumptions. The studies of weak-strong uniqueness of the weak solutions obtained in the above papers were performed in \cite{gala}. 

Let us comment that the question concerning regularity of the obtained solution is not clear. One cannot a priori exclude the finite-time blowup of first derivatives, see \cite{DafHsiao}.

Finally, let us mention that related system of thermoviscoelasticity was also extensively studied, see e.g. \cite{DafHsiao1,RackeZheng}, \cite[Chapter 12]{Roub} and references within. However, in this case the system is of parabolic - parabolic type and therefore analysis  of such system is much different form the analysis presented here.

\subsection{Linear thermoelasticity}
Under the assumption that $\theta$ is close to a state $\theta_0=const$, the system of equations $\eqref{system0}_{1,2}$ can be reduced to a linear one (see for instance \cite[Chapter 4]{racke2000evolution}):
\begin{equation}\label{linearTh}
\begin{aligned}
\begin{cases}
\mathbf{u}_{tt} = \nabla \cdot ( \nabla\mathbf{u}   -\mu \theta I),\\
\theta_t - \Delta \theta =- \mu \theta_0 \nabla\cdot \mathbf{u}_t.
\end{cases}
\end{aligned}
\end{equation}
This is a well-known system which has been studied extensively by many authors. To the best of our knowledge the first result was due to Dafermos \cite{D}, followed by plethora of other works, see e.g. \cite{racke2000evolution}, \cite[Chapter 7]{Lagnese}, \cite[Chapter 3]{LT00}, \cite{LZ98},  and references within.  

Let us point out some important properties of $\eqref{linearTh}$. First of all, the equation $\eqref{linearTh}_2$ is actually an entropy equation, the entropy is $\theta+\mu\theta_0\nabla\cdot\mathbf{u}$, the entropy flux is $-\nabla\theta$, while the entropy production is zero. To obtain the corresponding heat equation, one multiplies $\eqref{linearTh}_2$ with $\theta$ which gives
\begin{eqnarray*}
    \frac12 (\theta^2)_t - \frac12 \Delta (\theta^2)+|\nabla\theta|^2 = - \mu \theta_0 \theta \nabla\cdot \mathbf{u}_t.
\end{eqnarray*}
From a physical point of view, there are three major issues here - the entropy production is zero, there is a dissipation coming from temperature $|\nabla\theta|^2$ and the system allows temperature to be negative (the maximum principle isn't satisfied). Moreover, this model predicts that temperature will oscillate around (or simply stay close to) a stationary state, which is not physical, since mechanical energy tends to irreversibly transfer to heat, thus increasing it in the process.

While this linear system has nice mathematical properties and serves as a good model for the study of stability of thermoelastic bodies, it is physical relevant only in certain regimes under quite restrictive assumptions. There have been efforts throughout the last four decades to study more nonlinear models, but most of them are restricted to special 1D cases or the developed theory requires special forms of nonlinearities. In this paper, we are finally able to bridge this gap and introduce a meaningful notion of a weak solution for a model in nonlinear thermoelasticity, which is in full agreement with the laws of thermodynamics. Moreover, these solutions are global-in-time for large data, coincide with the regular ones (provided such exists) and are also consistent in a sense that if they are smooth, they satisfy the equations pointwise. Finally, let us point out that the results also hold in lower dimensions 1 and 2.

\section{Physical derivation}\label{Phys}

The present section is devoted to the physical derivation of our model as well as some related models. When deriving the model, we obtain candidates for natural variables in which a model should be examined. This way we are led to the formal version of the important estimates on the one hand, and to the natural framework in which the weak solution is defined on the other hand.

\subsection{Derivation of the model}\label{derivation}
We start with a system of equations modeling the balance of momentum and internal energy
\begin{eqnarray}\label{original:system}
     \begin{cases}
    \rho_S\mathbf{u}_{tt} = \nabla\cdot\boldsymbol\sigma,\\
    e_t +\nabla\cdot \mathbf{q}  = \boldsymbol\sigma\cdot\nabla \mathbf{u}_t,  \end{cases}
\end{eqnarray}
where the state variables are the displacement vector and the temperature $\mathbf{u},\theta$, respectively, and given functions are the density of the elastic body $\rho_S$, the stress tensor $\boldsymbol\sigma$, the internal energy $e$ and the internal energy flux $\mathbf{q}$, while $\boldsymbol\sigma\cdot\nabla \mathbf{u}_t$ represents the work done by the elastic body. We next assume
\begin{eqnarray*}
    \rho_S \equiv 1, \quad \boldsymbol\sigma=\boldsymbol\sigma(\theta,\nabla \mathbf{u}),\quad e=e(\theta,\nabla \mathbf{u}), \quad \mathbf{q} = \mathbf{q}(\theta,\nabla\theta).
\end{eqnarray*}
The free energy function $\psi=\psi(\theta,\mathbf{F})$, where $\psi:\mathbb{R}\times \mathbb{R}^{3\times 3}\to \mathbb{R}$, is defined as $\psi:=e-\theta s$, where $s$ is the entropy, while its Frech\'{e}t derivative w.r.t. $\mathbf{F}$ is denoted as $\psi_{\mathbf{F}}$. As a consequence of the Clausius-Duhem inequality (see e.g. \cite[Lemma 1.1]{racke2000evolution}), one must impose
\begin{eqnarray*}
\boldsymbol\sigma(\theta,\nabla \mathbf{u})=\psi_{\mathbf{F}}(\theta,\nabla \mathbf{u}),\quad s(\theta,\nabla \mathbf{u})=-\psi_\theta(\theta,\nabla \mathbf{u}).
\end{eqnarray*}
It is useful to write these relations in the following form of Gibbs' relations:
\begin{eqnarray}
    &&\theta s(\theta,\nabla \mathbf{u})_\theta =e(\theta,\nabla \mathbf{u})_\theta, \label{gibbs1}\\
    &&\theta s(\theta,\nabla \mathbf{u})_{\mathbf{F}} =e(\theta,\nabla \mathbf{u})_{\mathbf{F}} -\boldsymbol\sigma(\theta,\nabla \mathbf{u}). \label{gibbs2}
\end{eqnarray}
 Dividing the internal energy equation $\eqref{original:system}_2$ by $\theta$ (provided that $\theta>0$), we arrive at the entropy equation
\begin{eqnarray}\label{original:entropy}
    s_t +\nabla \cdot \left(\frac{\mathbf{q}}{\theta} \right) = -\frac{\mathbf{q}\cdot\nabla\theta}{\theta^2},
\end{eqnarray}
where $\frac{\mathbf{q}}{\theta}$ is the entropy flux and $-\frac{\mathbf{q}\cdot\nabla\theta}{\theta^2}$ is the entropy production. In accordance with the second law of thermodynamics, we have to impose a non-negative entropy production
\begin{eqnarray}\label{positive:ent:prod}
    \mathbf{q}\cdot\nabla\theta \leq 0.
    \end{eqnarray}
The total energy balance can be obtained by multiplying $\eqref{original:system}_1$ with $\mathbf{u}_t$ and summing up with $\eqref{original:system}_2$ to obtain
\begin{eqnarray}\label{original:energy:dif}
    \frac{d}{dt}\mathcal{E}+ \nabla \cdot(\boldsymbol\sigma \mathbf{u}_t+\mathbf{q})=0,
\end{eqnarray}
where the total energy is
\begin{eqnarray*}
    \mathcal{E}=\frac{1}{2}|\mathbf{u}_t|^2 + e(\theta,\nabla \mathbf{u}).
\end{eqnarray*}
Now, assuming the following boundary conditions on $\partial\Omega$
\begin{eqnarray}
    (\boldsymbol\sigma \mathbf{u}_t)\cdot \mathbf{n}&=&0, \label{bnd:cond1}\\
    \mathbf{q}\cdot \mathbf{n}&=&0  \label{bnd:cond2},
\end{eqnarray}
we integrate (\ref{original:energy:dif}) over $(0,t)\times\Omega$ to obtain
\begin{eqnarray}\label{original:energy}
    \int_\Omega \mathcal{E}(t)=\int_\Omega\mathcal{E}(0),
\end{eqnarray}
for any $t\in (0,T]$, so the total energy is conserved.\\ ${}$\\

In order to simplify the model, we first assume that energy can be decomposed as
\begin{eqnarray}\label{en:dec}
    e(\theta,\nabla \mathbf{u})=e_1(\theta)+e_2(\nabla \mathbf{u}),
\end{eqnarray}
so $(\ref{gibbs1})$ implies
\begin{eqnarray*}
        s(\theta,\nabla \mathbf{u})_\theta = \frac{ e_1(\theta)}{\theta},
\end{eqnarray*}
which by integrating w.r.t. $\theta$ one concludes that entropy can also be decomposed
\begin{eqnarray}\label{ent:dec}
    s(\theta,\nabla \mathbf{u})=s_1(\theta)+s_2(\nabla \mathbf{u}),
\end{eqnarray}
and the relations $(\ref{gibbs1})$ and $(\ref{gibbs2})$ then simplify to
\begin{eqnarray}
    &&\theta s_1(\theta)_\theta =e_1(\theta)_\theta, \label{gibbs1:s}\\
    &&\theta s_2(\nabla\mathbf{u})_{\mathbf{F}} =e_2(\nabla\mathbf{u})_{\mathbf{F}} -\boldsymbol\sigma(\theta,\nabla\mathbf{u}). \label{gibbs2:s}
\end{eqnarray}
The original system (\ref{original:system}) now simplifies to
\begin{eqnarray}\label{original:system:s}
    \begin{cases}
    \mathbf{u}_{tt}=\nabla\cdot(e_2(\nabla \mathbf{u})_{\mathbf{F}}-\theta s_2(\nabla \mathbf{u})_{\mathbf{F}}),\\
    e_1(\theta)_t +\nabla\cdot \mathbf{q}(\theta,\nabla\theta) =-\theta s_2(\nabla \mathbf{u})_t,
    \end{cases}
\end{eqnarray}
while the entropy equation stays the same. Note that the internal energy equation $(\ref{original:system})_2$ now becomes the heat equation $(\ref{original:system:s})_2$, since the elastic energy $e_2(\nabla \mathbf{u})$ is canceled out.
\begin{rem}
Let us point out an important observation. In the above system, due to decoupling of the energy $\eqref{en:dec}$, the stress tensor is of the form
\begin{eqnarray*}
    \boldsymbol\sigma(\theta,\nabla\mathbf{u})=e_2(\nabla\mathbf{u})_{\mathbf{F}}-\theta s_2(\nabla\mathbf{u})_{\mathbf{F}},
\end{eqnarray*}
while the work transfered to heat becomes $-\theta s_2(\nabla \mathbf{u})_t$. This is very restrictive in $\theta$ and completely determines the coupling.
\end{rem}
Finally, the following assumptions of thermodynamical stability are assumed on the energy and entropy:
\begin{eqnarray}
       &&e_1(\theta),~e_1(\theta)_\theta>0, \quad \mbox{for all } \theta>0,\label{stability1a}\\
       &&\lim\limits_{\theta\to 0^+} s_1(\theta) = -\infty, \label{stability1b}
\end{eqnarray}
and there exist constants $c_1>0$ and $c_2$ such that\footnote{This assumption is automatically satisfied if, for example, $e_2$ is convex and $s_2$ is concave.}
\begin{eqnarray}\label{stability2}
      e_2(\mathbf{F})-c_1 s_2(\mathbf{F})\geq c_2,
\end{eqnarray}
for all $\mathbf{F}\in \mathbb{R}^{3\times 3}$. Now, for any constant $\overline{\theta}>0$, we introduce the Helmholtz function
\begin{eqnarray*}
       H_{\overline{\theta}}(\theta,\nabla \mathbf{u}):=e(\theta,\nabla \mathbf{u}) - \overline{\theta}s(\theta,\nabla \mathbf{u}).
\end{eqnarray*}
Now, we multiply $(\ref{original:entropy})$ with $\overline\theta$, integrate over $(0,T)\times\Omega$ and then subtract from $(\ref{original:energy})$, which then gives us the total dissipation balance
\begin{eqnarray}\label{total:dissipation:balance}
     &&\int_\Omega H_{\overline{\theta}}(t) +  \frac12 \int_\Omega |\mathbf{u}_t|^2(t) + \overline{\theta} \int_0^t\int_\Omega -\frac{\mathbf{q}(\theta,\nabla\theta)\cdot\nabla\theta}{\theta^2}  \nonumber \\
   &&= \int_\Omega H_{\overline{\theta}}(0)+ \frac12 \int_\Omega |\mathbf{v}_0|^2.
\end{eqnarray}
First, by choosing $\overline{\theta}=c_1$,  from \eqref{stability1a}, \eqref{stability1b} and \eqref{stability2}, one directly has that $\theta>0$, at least a.e. Next, noticing that the condition (\ref{stability1a}) implies $s_1(\theta)_\theta > 0$, one has by (\ref{gibbs1:s})
\begin{eqnarray*}
    (e_1 - c_1 s_1)_\theta=(e_1)_\theta - c_1(s_1)_\theta = \underbrace{(e_1)_\theta - \theta (s_1)_\theta}_{=0}+(\theta-c_1)(s_1)_\theta.
\end{eqnarray*}
Therefore the function $e_1 - c_1 s_1$ attains its global minimum on $\mathbb{R}^+$ at $\theta=c_1$. This combined with $(\ref{stability2})$ implies the coercivity of the Helmholtz function
\begin{eqnarray}\label{HelmholtzCo}
      H_{c_1}(\theta,\nabla \mathbf{u})=e(\theta,\nabla \mathbf{u}) - c_1 s(\theta,\nabla \mathbf{u}) \geq C,
\end{eqnarray}
where $C$ only depends on $c_1$ and $c_2$ from $(\ref{stability2})$. Thus, (\ref{total:dissipation:balance}) for $\overline{\theta}=c_1$ implies the boundedness of the entropy production.

\subsection{Prescribing the constitutive functions}
First, we will assume that the deformation is small, which leads to the following assumption
\begin{eqnarray*}
     e_2(\nabla \mathbf{u}) = \frac12|\nabla\mathbf{u}|^2.
\end{eqnarray*}
Next, it is standard to assume that the stress due to heat expansion is of the following linear form $\theta s_2(\nabla \mathbf{u})_{\mathbf{F}}=\mu \theta I$, where $\mu>0$, which directly leads to 
\begin{eqnarray*}
    s_2(\nabla \mathbf{u}) = \mu\nabla\cdot \mathbf{u},
\end{eqnarray*}
since $(\nabla\cdot \bb{u})_{\bb{F}}=I$. Note that $\nabla\cdot \bb{u}$ represents the linearization of Jacobian and thus approximates the volume, which is consistent with the natural increase of the entropy with the volume. Next, the internal energy flux satisfies the Fourier law
\begin{eqnarray*}
     \bb{q}(\theta,\nabla\theta)=-\kappa(\theta)\nabla\theta,
\end{eqnarray*}
where $\kappa(\theta)>0$ is the heat conductivity. Finally, we assume that the body is fixed at the boundary and the internal energy flux through the boundary is zero:
\begin{eqnarray*}
    && \mathbf{u} = 0, \quad \partial_n\theta=0, \quad \mbox{on } (0,T)\times\partial\Omega.
\end{eqnarray*}
Supplementing with initial data, this leads to the following problem:
\begin{eqnarray*}
  \begin{cases}
  \bb{u}_{tt}=\nabla\cdot(\nabla\bb{u}  - \mu\theta I),&\quad \text{ in } (0,T)\times\Omega,\\
  c_V(\theta) \theta_t -  \nabla\cdot(\kappa(\theta)\nabla \theta) = -\mu \theta \nabla\cdot \bb{u}_t,&\quad \text{ in } (0,T)\times\Omega,\\
  \bb{u}=0,~ \partial_n \theta = 0,&\quad \text{on } (0,T)\times \partial\Omega,\\
  \bb{u}(0,\cdot)=\bb{u}_0, ~  \bb{u}_t(0,\cdot)=\bb{v}_0, ~ \theta(0,\cdot)=\theta_0,&
  \end{cases}
\end{eqnarray*}    
where $c_V(\theta):=e_1(\theta)_\theta$ is the heat capacity. It remains to prescribe $c_V(\theta)$ and $\kappa(\theta)$.
\subsubsection{High temperature model}
For many homogeneous materials, the heat capacity is known to be almost constant at high temperatures due to the Dulong-Petit law (see \cite[Section 2]{oxword}). On the other hand, heat conductivity tends to show little change at high temperatures in many relevant materials\footnote{For example crystals and glasses.} (see \cite{ross}). This leads to the first model:
\begin{eqnarray*}
    c_V(\theta) = 1, \quad \kappa(\theta)=1.
\end{eqnarray*}

\subsubsection{Low temperature model}\label{non:lin:const}
At low temperatures, the heat capacity  for non-magnetic materials increases as $\theta^3$ due to the Debye law (see \cite[Section 2.2]{oxword}). Moreover, the heat conductivity at low temperatures also tends to increase as $\theta^3$ (see \cite[Pages 504, 505]{solidstate}). To keep the result general, the following is chosen:
\begin{eqnarray}
    c_V(\theta) = 1+\alpha\theta^{\alpha-1}, \quad \kappa(\theta)=1+\theta^\beta,
\end{eqnarray}
where $\alpha>1$ and $\beta>0$. Here, the constant part was added due to mathematical reasons. In particular, the constant in $c_V(\theta)$ ensures that $\ln\theta$ will be bounded, so the temperature will not vanish a.e, while the constant part in $\kappa(\theta)$ is added to provide enough control over $\ln\theta$. The conditions $\alpha>1$ and $\beta>0$ are  chosen so that $c_V(\theta)$ and $\kappa(\theta)$ are growing functions.

\section{Proof of Theorem \ref{main:lin}}
In our approach, it is crucial to work with thermodynamical variables, i.e. to use equation \eqref{ent:eq1}. However, this equation is formally equivalent with \eqref{system0}$_2$ under assumption $\theta>0$, which needs to be proved. Therefore, we split our proof in two parts. In the first part, we construct approximate solutions $(\vu^n,\theta^n)$ to equation \eqref{system0} such that $\theta^n>0$. In the second step, we introduce thermodynamical variables and pass to the limit in approximations parameters to obtain a weak solution in sense of Definition \ref{weaksolution}.

\subsection{Step 1: Approximate problem}
We introduce a smooth basis $\{\vphi_i\}_{i\in \mathbb{N}}$ of $H_0^1(\Omega)^3$ and denote $V_n:={\rm span}\{\vphi_i\}_{1\leq i\leq n}$. The approximate problem is defined as follows.
\begin{mydef}[Approximate problem]\label{DefApp}
We say that $(\vu^n,\theta^n)\in W^{1,\infty}(0,T;V_n)\times L^2(0,T;H^1(\Omega))$ is a solution to the approximate problem if for all $\vphi\in V_n$ and $\phi\in H^1(\Omega)$ the following equations are satisfied in $\mathcal{D}'(0,T)$
\begin{eqnarray}\label{APP1}
      \frac{d^2}{dt^2} \int_\Omega \mathbf{u}^n\cdot \vphi+\int_\Omega \nabla\mathbf{u}^n:\nabla \vphi=  \mu\int_\Omega \theta^n \nabla \cdot \vphi,
      \\\label{APP2}
       \frac{d}{dt}\int_{\Omega}\theta^n\phi+\int_{\Omega}\nabla \theta^n\cdot\nabla\phi=-\mu\int_{\Omega}\theta^n ( \nabla\cdot\mathbf{u}^n_t)\phi.
\end{eqnarray}
Moreover, for the initial data, we choose $0< c_n\leq \theta^n(0)=\theta_0^n\in H^1(\Omega)$ as a regularization of $\theta_0$ such that $\partial_n \theta_0^n =0$ on $\partial\Omega$, $\theta_0^n\to \theta_0$ in $L^1(\Omega)$ as $n\to \infty$ and $\int_\Omega \theta_0^n \leq \int_\Omega \theta_0$ for all $n\in \mathbb{N}$. For the displacement, we choose $\mathbf{u}^n(0)=P_{V_n}\vu_0$, $ \mathbf{u}_t^n(0)=P_{V_n}\mathbf{v}_0$, where $P_{V_n}$ is an orthogonal projection onto $V_n$.
\end{mydef}

\begin{rem}
Note that we are actually using the heat equation in approximate problem, while later we will switch to the entropy equation in order to pass to the limit $n\to\infty$.
\end{rem}

\begin{lem}\label{EnergyApprox}
Let $(\vu^n,\theta^n)\in W^{1,\infty}(0,T;V_n)\times L^2(0,T;H^1(\Omega))$ be a solution to the approximate problem given by Definition \ref{DefApp}. Then the following energy equality is satisfied:
\begin{align}\label{EEApp}
       \int_\Omega\theta^n(t)+\frac12\int_\Omega |\vu^n_t(t)|^2+\frac12 \int_\Omega |\nabla\vu^n(t)|^2=E^n_0,\quad \text{ for  a.a. } t\in [0,T],
\end{align}
where $E^n_0=\int_{\Omega}\theta_0^n+\frac12\int_\Omega|P_{V_n}\mathbf{v}_0|^2+\frac12\int_{\Omega}|\nabla P_{V_n}\vu_0|^2$ is the energy of the initial data projected onto finite-dimensional subspace $V_n$.
\end{lem}
\begin{proof}
We take $\vphi=\vu_t^n$ in \eqref{APP1}, $\phi=1$ in \eqref{APP2}, add the resulting equations and integrate over $(0,t)$. Notice that the terms on the right hand sides cancel.
\end{proof}

First we prove a version of the standard maximum principle result for parabolic equations which we will use to prove strict positivity of the temperature.
\begin{lem}\label{MaxPrinciple}
Let $a\in L^1(0,T;L^{\infty}(\Omega))$ and $\theta\in L^2(0,T;H^2(\Omega))\cap H^1(0,T;L^2(\Omega))$ be the solution to the following parabolic problem:
\begin{align*}
 \begin{cases}
    \theta_t-\Delta\theta=-a\theta,&\quad  \text{ in } (0,T)\times\Omega,
    \\
    \partial_n\theta=0,&\quad \text{ on } (0,T)\times\partial\Omega,
    \\
    \theta(0)=\theta_0^n>0.&
    \end{cases}
\end{align*}
Then
\begin{align*}
    \theta(t,\vx)\geq \left (\min_{\vx\in\Omega}\theta_0^n(\vx)\right)
    \exp\left(-\int_0^t\|a(s)\|_{L^{\infty}(\Omega)}ds\right ).
\end{align*}
\end{lem}
\begin{proof}
We define $C_0=\min_{\vx\in\Omega}\theta_0^n(\vx)$ and
\begin{align*}
    d(t,\vx)=\theta(t,\vx)-C_0\exp\left(-\int_0^t\|a(s)\|_{L^{\infty}(\Omega)}ds\right ).
\end{align*}
By straightforward calculation $d$ satisfies the following equation:
\begin{align*}
    d_t-\Delta d+ad=C_0\exp\left(-\int_0^t\|a(s)\|_{L^{\infty}(\Omega)}ds\right )\underbrace{(\|a\|_{L^{\infty}}-a)}_{\geq 0}.
\end{align*}
Therefore
\begin{align}\label{dEq}
     d_t-\Delta d+ad\geq 0,\quad d(0)\geq 0.
\end{align}
We multiply \eqref{dEq} with $d^-:=\max\{-d,0\}$ to get:
\begin{align*}
    \frac{1}{2}\frac{d}{dt}\|d^-\|^2_{L^2(\Omega)}+\|\nabla d^-\|^2_{L^2(\Omega)}
    \leq -a\|d^-\|_{L^2(\Omega)}^2.
\end{align*}
Since $d(0)\geq 0$, we have $d^-(0)=0$ and therefore by the Gronwall inequality we get $d^-=0$ which finishes the proof.
\end{proof}

\begin{proposition}\label{ExistenceApp}
For every $n\in\N$ there exists a solution $(\mathbf{u},\theta)$ to the approximate problem in the sense of Definition \ref{DefApp}. Moreover, one has $\theta\in L^2(0,T;H^2(\Omega))\cap H^1(0,T;L^2(\Omega))$.
\end{proposition}
\begin{proof}
We prove existence by using Schaefer's fixed point theorem. 
We define the fixed point set:
$$
X_n= C^1([0,T];V_n).
$$
We equip $X_n$ with the following norm: $\|\vu\|_{X_n}=\max_{t\in [0,T]}\{\|\vu(t)\|_{L^2(\Omega)},\|\vu_t(t)\|_{L^2(\Omega)}\}$.
The fixed point mapping $\mathcal{A}$  is defined in the following way. Let $\vu\in X_n$. We define $\theta(\vu)$ by solving \eqref{APP2} with $\vu^n$ on the right hand side. More precisely, $\theta(\vu^n)$ satisfies the following equation for every $\phi\in H^1(\Omega)$:
\begin{align}\label{DefAppTheta}
 \frac{d}{dt}\int_{\Omega}\theta(\vu)\phi+\int_{\Omega}\nabla \theta(\vu)\cdot\nabla\phi=-\mu\int_{\Omega}\theta(\vu) ( \nabla\cdot\mathbf{u}_t)\phi.
\end{align}
Existence of a weak solution $\theta(\vu)\in L^2(0,T;H^1(\Omega))$ is a consequence of the standard parabolic theory. Moreover, by taking $\phi=\theta(\vu)$, we obtain the following estimate:
\begin{align}\label{EstAppTheta}
    \frac12\frac{d}{dt}\|\theta(\vu)\|^2_{L^2(\Omega)}
    +\|\nabla\theta(\vu)\|^2_{L^2(\Omega)}
    \leq \mu\|\nabla\cdot\vu_t\|_{L^{\infty}(\Omega)}\|\theta(\vu)\|^2_{L^2(\Omega)}.
\end{align}
However, since $V_n$ is a finite dimensional space, all norms are equivalent so we also have
\begin{align}\label{EstimateFD}
     \|\nabla\cdot\vu_t\|_{L^{\infty}(\Omega)}
     \leq C_n\|\vu_t\|_{L^2(\Omega)}.
\end{align}
Note that the constant in \eqref{EstimateFD} depends on $n$, but in the proof of this Lemma $n$ is fixed so this dependence does not influence the proof. By combining \eqref{EstAppTheta} and \eqref{EstimateFD}, and using Gronwall's inequality we get:
\begin{align}\label{EstimateAppTheta}
    \|\theta(\vu)\|_{L^2(0,T;H^1(\Omega))}\leq C_n(\|\vu\|_{X_n}),
\end{align}
where $C_n:\R_+\to\R_+$ is a monotone function. Now, $\mathcal{A}\vu$ is obtained by solving \eqref{APP1} with $\theta(\vu)$ instead of $\theta$ on the right hand side. More precisely, $\mathcal{A}\vu$ satisfies the following equations:
\begin{align}\label{AppWave}
      \frac{d^2}{dt^2} \int_\Omega (\mathcal{A}\vu)\cdot \vphi+\int_\Omega \nabla(\mathcal{A}\vu):\nabla \vphi=  \mu\int_\Omega \theta(\vu) \nabla \cdot \vphi,\quad \vphi\in V_n.
\end{align}
Existence of $\mathcal{A}\vu\in C^1([0,T];V_n)$ is obtained by solving a linear ODE system \eqref{AppWave}. Therefore the operator $\mathcal{A}:X_n\to X_n$ is well defined. Let us prove that is satisfies the assumptions of Schaefer's fixed point theorem.

\noindent
\textbf{Continuity of $\mathcal{A}$}.
Let $\vu^m$ be a convergent sequence in $X_n$ and $\vu=\lim\limits_{m\to\infty} \vu^m$. Let $\psi^m:=\theta(\vu)-\theta^m(\vu)$. By \eqref{DefAppTheta}, $\psi^m$ satisfies the following equality with zero initial data:
\begin{align*}
    &\frac{d}{dt}\int_{\Omega}\psi^m\phi+\int_{\Omega}\nabla \psi^m\cdot\nabla\phi\\
    &=-\mu\int_{\Omega}\theta(\vu) ( \nabla\cdot\mathbf{u}_t-\nabla\cdot\vu_t^m)\phi
    -\mu\int_{\Omega}(\nabla\cdot\vu_t^m)\psi^m\phi.
\end{align*}
By taking $\phi=\psi^m$, we arrive at
\begin{align*}
    &\frac{1}{2}\frac{d}{dt}\|\psi^m\|^2_{L^2(\Omega)}
    +\|\nabla\psi^m\|^2_{L^2(\Omega)}\\
    &\leq
    \mu\| \nabla\cdot(\mathbf{u}_t-\vu_t^m)\|_{L^{\infty}(\Omega)}\|\theta(\vu)\|_{L^2(\Omega)}\|\psi^m\|_{L^2(\Omega)}\\
    &\quad+\mu\|\psi^m\|^2_{L^2(\Omega)}\|\nabla\cdot \vu_t^m\|_{L^{\infty}(\Omega)}.
\end{align*}
Next, \eqref{EstimateAppTheta} and boundedness of $\|\nabla\cdot \vu_t^m\|_{L^{\infty}(\Omega)}$, together with Gronwall's inequality, yield a bound 
\[
\|\nabla\psi^m\|^2_{L^\infty(0,T;L^2(\Omega))}\leq C_n\|\nabla\cdot(\mathbf{u}_t-\vu_t^m)\|_{L^{\infty}(0,T;L^{\infty}(\Omega))}.
\]
This, in turn, allows us to obtain:
\begin{align}\label{EstimateContTh}
    \|\psi^m\|_{L^2(0,T;H^1(\Omega))}
    \leq C_n\|\nabla\cdot(\mathbf{u}_t-\vu_t^m)\|_{L^{\infty}(0,T;L^{\infty}(\Omega))},
\end{align}
where $C_n$ is a constant independent of $m$. Now, let us denote $\mathbf{a}^m=\mathcal{A}\vu-\mathcal{A}\vu^m$. By definition of operator $\mathcal{A}$ in \eqref{AppWave}, $\mathbf{a}^m$ satisfy the following equations with zero  initial data:
\begin{align*}
    \frac{d^2}{dt^2} \int_\Omega \mathbf{a}^m\cdot \vphi+\int_\Omega \nabla \mathbf{a}^m:\nabla \vphi=  \mu\int_\Omega \psi^m \nabla \cdot \vphi,\;\vphi\in V_n.
\end{align*}
By taking $\vphi=\mathbf{a}_t^m$, we get
\begin{align*}
    \frac12\frac{d}{dt}\left (\| \mathbf{a}_t^m\|^2_{L^2(\Omega)}
    +\|\nabla\mathbf{a}^m\|^2_{L^2(\Omega)}\right )
    &\leq\|\nabla\psi^m\|_{L^2(\Omega)} \| \mathbf{a}_t^m\|_{L^2(\Omega)}
    \\
    &\leq  C_n\|\nabla\cdot(\mathbf{u}_t-\vu_t^m)\|_{L^{\infty}(0,T;L^{\infty}(\Omega))}
    \|\mathbf{a}_t^m\|_{L^2(\Omega)}.
\end{align*}
Since all norms on $V_n$ are equivalent we have $\|\nabla\cdot(\mathbf{u}_t-\vu_t^m)\|_{L^{\infty}(0,T;L^{\infty}(\Omega))}\to 0$ and therefore, by Gronwall's and Young's inequalities, $\|\mathbf{a}^m\|_{X_n}\to 0$. Thus we proved that $\mathcal{A}$ is a continuous function on $X_n$.

\noindent
\textbf{Compactness of $\mathcal{A}$}.
Let $\vu\in X_n$, $\|\vu\|_{X_n}\leq R$.
By \eqref{AppWave}, we have
\begin{align*}
    &\|(\mathcal{A}\vu)_{tt}\|_{L^2(0,T;L^2(\Omega))} \\
    &\leq \|\Delta(\mathcal{A}\vu)\|_{L^2(0,T;L^2(\Omega))}
    +\mu\|\nabla(\theta(\vu))\|_{L^2(0,T;L^2(\Omega))}
    \leq C_n(R).
\end{align*}
Here we have used continuity of $\mathcal{A}$ and equivalence of norms on $X_n$ to bound the first term, and \eqref{EstimateAppTheta} to bound the second. Since $X_n$ is finite dimensional, these estimates immediately give compactness of operator $\mathcal{A}$ due to the compactness of the embedding $H^1(0,T)\hookrightarrow C[0,T]$.

\noindent
\textbf{Boundedness of $Y=\{\vu\in X_n:(\exists\lambda\in (0,1])\;\vu=\lambda\mathcal{A}\vu\}$.}

Let $\vu\in Y$ and $\lambda\in (0,1]$ such that $\vu=\lambda\mathcal{A}\vu$. Then $(\frac{1}{\lambda}\vu,\theta(\vu))$ is a solution to the approximate problem given by Defintion \ref{DefApp}. Therefore, the following energy equality is satisfied by Lemma \ref{EnergyApprox}:
\begin{align*}
       &\int_\Omega\theta(\vu)(t)+\frac12\frac{1}{\lambda}\int_\Omega |\vu_t(t)|^2+\frac12\frac{1}{\lambda} \int_\Omega |\nabla\vu(t)|^2
       \\
       &=\int_{\Omega}\theta_0^n+\frac12\frac{1}{\lambda}\int_\Omega|P_{V_n}\mathbf{v}_0|^2+\frac12\frac{1}{\lambda}\int_{\Omega}|\nabla P_{V_n}\mathbf{u}_0|^2,\quad t\in [0,T].
\end{align*}
Therefore,
\begin{align*}
    \|\vu\|_{C^1([0,T;V_n)}\leq CE_0^n,\quad \vu\in Y.
\end{align*}

Thus, by Schaefer's fixed point theorem there exists a fixed point $\vu$ of the mapping $\mathcal{A}$. Moreover, by Lemma \ref{MaxPrinciple}, $\theta(\vu)$ is a strictly positive function. Therefore, $(\vu,\theta(\vu))$ is a solution to the approximate problem in the sense of Definition \ref{DefApp}.

Finally, since $\theta \nabla\cdot \mathbf{u}_t \in L^2((0,T)\times \Omega)$, there is a unique solution $\tilde\theta \in L^2(0,T;H^2(\Omega))\cap H^1(0,T;L^2(\Omega))$ to the equation
\begin{eqnarray*}
    \tilde\theta_t-\Delta\tilde\theta=-\mu\theta \nabla\cdot \mathbf{u}_t,
\end{eqnarray*}
so $\theta\equiv \tilde\theta$ and the improved regularity is thus obtained.
\end{proof}

\subsubsection{Passing to the limit}
\begin{lem}\label{WeakConvergenceLemma}
There exists $\theta\in L^{\infty}(0,T;\mathcal{M}^+(\Omega))$ and $\vu\in L^{\infty}(0,T;H^1(\Omega))\cap W^{1,\infty}(0,T;L^2(\Omega))$ such that the following convergence properties hold (on a subsequence)
\begin{eqnarray*}
    &\theta^n \rightharpoonup \theta,& \quad \text{ weakly* in } L^{\infty}(0,T;\mathcal{M}^+(\Omega)),
    \\
    &\vu^n \rightharpoonup \vu,& \quad \text{ weakly* in } L^{\infty}(0,T;H^1(\Omega)),
    \\
    &\vu_t^n \rightharpoonup \vu_t,& \quad \text{ weakly* in } L^{\infty}(0,T;L^2(\Omega)).
\end{eqnarray*}
\end{lem}
\begin{proof}
As a direct consequence of energy equality given in Lemma \ref{EnergyApprox}, the following uniform estimate holds:
\begin{align}\label{UnformEnergyEstimate}
    \|\theta^n\|_{L^{\infty}(0,T;L^1(\Omega))}
    +\|\vu\|_{{L^{\infty}(0,T;H^1(\Omega))}}
    +\|\vu_t\|_{{L^{\infty}(0,T;L^2(\Omega))}}
    \leq E_0,
\end{align}
where $n\in\mathbb{N}$ and $E_0=\|\theta_0\|_{L^1(\Omega)}     +\|\nabla\vu_0\|_{L^2(\Omega)}     +\|\mathbf{v}_0\|_{L^2(\Omega)}$. Therefore, the existence of a weakly* convergence subsequence follows from standard results about weak* compactness of a ball in the corresponding function spaces.
\end{proof}
By using Lemma \ref{WeakConvergenceLemma} we can pass to the limit in equation \eqref{APP1} and prove that $(\vu,\theta)$ satisfy the following variational equation:
\begin{align*}
    -\int_0^T\int_\Omega\mathbf{u}_t\cdot \vphi_t+\int_0^T\int_\Omega \nabla\mathbf{u}:\nabla \vphi=  \mu\int_0^T\int_\Omega \nabla \cdot \vphi d\theta,
\end{align*}
for all $\vphi\in C^{\infty}_0((0,T)\times\Omega)$. However, weak* convergences from Lemma \ref{WeakConvergenceLemma} are not enough for passing to the limit in \eqref{APP2} because it contains nonlinear terms. Since approximate solutions are smooth and $\theta^n\geq C_n>0$, $n\in\N$ (see Lemma \ref{MaxPrinciple}), we can divide the heat equation $\eqref{APP2}$ by $\theta^n$, so by using a new unknown $\tau^n=\ln\theta^n$, we can rewrite \eqref{APP2} in the form of entropy equation:
\begin{align}\label{AppEntropy}
    &\int_0^T \int_\Omega \tau^n \phi_t
    -\int_0^T \int_\Omega \nabla \tau^n \cdot \nabla \phi
    +\mu\int_0^T \int_\Omega \mathbf{u}^n_t \cdot \nabla \phi+\int_0^T\int_{\Omega}|\nabla\tau^n|^2\phi   \nonumber\\
    & =-\int_{\Omega}\tau_0\phi(0,x),\qquad \phi\in  C_0^\infty([0,T)\times \Omega),
\end{align}
where $\tau_0^n=\ln\theta_0^n$. Now, one can also show that the total dissipation balance  $\eqref{total:dissipation:balance:lin}$ holds for the approximate solutions, so we get the following additional estimate:
\begin{align}\label{ToDissEst}
    \|\nabla\tau^n\|_{L^{2}(0,T;L^2(\Omega))}
    +\|\tau^n\|_{L^{\infty}(0,T;L^1(\Omega))}\leq C.
\end{align}
Therefore there exists $\tau\in L^2(0,T;H^1(\Omega))$ and measure $\boldsymbol\sigma\in\mathcal{M}^+([0,T]\times\overline{\Omega})$ such that
\begin{align*}
    \tau_n\rightharpoonup \tau& \quad{\rm weakly}\;{\rm in}\; L^{2}(0,T;H^1(\Omega)),
    \\
    |\nabla\tau^n|^2\rightharpoonup \boldsymbol\sigma &\quad{\rm weakly*}\;{\rm in}\;\mathcal{M}^+([0,T]\times\overline{\Omega}).
\end{align*}
Using these convergence properties we can pass to the limit in \eqref{APP2} to obtain the entropy equation \eqref{entweak}. It remains to identify relations between $\tau$, $\theta$ and $\boldsymbol\sigma.$ First, we prove $\boldsymbol\sigma\geq|\nabla\tau|^2$. For $\phi\in  C([0,T]\times\overline{\Omega})$, we define mapping $I_{\phi}:L^2(\Omega)\to\R$ with
$$
I_{\phi}(f)=\int_{\Omega}f^2\phi.
$$
Notice that this is a continuous mapping and convex for $\phi\geq 0$. In particular, it is lower semicontinuous in strong $L^2$ topology. Therefore, by Mazur's lemma (see e.g. \cite[Thm. 9.1]{BC-CA}), $I_{\phi}$ is weakly lower semicontiunous. Thus we have:
$$
\int_{\Omega}|\nabla\tau|^2\phi=I_{\phi}(\nabla\tau)
=I_{\phi}(\lim_n\nabla\tau^n)
\leq \lim_n I_{\phi}(\nabla\tau^n)=\langle\boldsymbol\sigma,\phi\rangle,
$$
which gives $\boldsymbol\sigma\geq|\nabla\tau|^2$
.

From equation \eqref{AppEntropy}, we immediately get that time derivatives of $\tau^n$ are uniformly bounded in $L^1(0,T;W^{-1,p^*}(\Omega))$, $p>3$,  i.e. we have the following estimate:
$$
\|\tau_t^n\|_{L^1(0,T;W^{-1,p^*}(\Omega))}\leq C.
$$
Combining this estimate with \eqref{ToDissEst} and Aubin-Lions lemma \cite{Aubin} we get that sequence $\tau^n$ strongly converges (after passing to a subsequence) in $L^1((0,T)\times\Omega)$.

Therefore, $\theta^n$ converges almost everywhere (on a subseqence), i.e. we have
$$
\ln\theta^n=\tau^n\to \tau\quad {\rm a.e.}\; {\rm in}\; (0,T)\times\Omega,
$$
which also implies
$$
\theta^n = e^{\tau_n}\to e^\tau \quad {\rm a.e.}\; {\rm in}\; (0,T)\times\Omega.
$$
We can now use the theorem of Egorov, which gives us that for every $\varepsilon>0$ there is a set $A_\varepsilon \subset (0,T)\times \Omega$ such that $|((0,T)\times \Omega)\setminus A_\varepsilon|<\varepsilon$ and $\theta^n \to e^\tau$ uniformly on $A_\varepsilon$, and in light of uniqueness of uniform and distributional limits, the continuous part of $d\theta$ equals $e^\tau$ on $A_\varepsilon$. By letting $\varepsilon\to 0$, one obtains that the continuous part of $\theta$ is $e^{\tau}$ a.e. on $(0,T)\times\Omega$. Note that we do not rule out that $\theta$ also includes a singular part supported on sets of Lebesgue measure $0$ and therefore cannot conclude $d\theta=e^{\tau}$. Nevertheless, we will prove that inequality $e^{\tau}\leq d\theta$ holds in the sense of measures. Let $E$ be a set of all points where $\theta^n$ does not converge to $\theta$ and $E_t=\{x\in\overline{\Omega}:(t,x)\in E\}$. $E$ is a set of measure zero in $(0,T)\times\Omega$ and thus $E_t$ is a set of measure zero in $\Omega$ for almost every $t\in (0,T)$, and singular part of measure $\theta$ (we name it $f$) is supported in $E$. Now by Fatou's lemma, for any non-negative $\phi\in  C(\overline{\Omega})$ and almost every $t\in (0,T)$ we have:
$$
\int_{\Omega}e^{\tau}(t)\phi
=\int_{\Omega\setminus E_t}e^{\tau}(t)\phi
=\int_{\Omega\setminus E_t}\lim_n\theta^n(t)\phi
\leq \lim_n\int_{\Omega\setminus E_t}\theta^n(t)\phi
=\int_{\Omega}\phi d\theta(t).
$$
Therefore, for any non-negative $\phi\in  C(\overline{\Omega})$ and almost every $t\in (0,T)$ we have
$$
\int_{\Omega}e^{\tau}(t)\phi\leq \int_{\Omega}\phi d\theta(t).
$$
This concludes the proof of the existence part of Theorem \ref{main:lin}.

\subsection{Consistency}
Assume that $(\mathbf{u},\tau)$ is a weak solution in the sense of Definition $\ref{weaksolution}$ which is smooth. The goal  is to prove that then $(\mathbf{u},e^\tau)$ is a classical solution to the system $\eqref{system0}$. First, it is obvious that $\eqref{system0}_1$ is satisfied and $\mathbf{u}=0$ on $(0,T)\times\partial\Omega$. Moreover, since the momentum equation is satisfied pointwise, this implies that $\theta$ is a function so $e^\tau=\theta$. Now, for any $t\in(0,T]$, let $\psi_n\in C(\mathbb{R})$ be a sequence of functions defined as
\begin{eqnarray*}
    \psi_n(\tau):=\begin{cases}
    1,& \text{for } \tau\leq t,\\
    1-n(\tau-t),& \text{for } t<\tau\leq t+\frac1n,\\
    0,& \text{for } \tau >t+\frac1n,
    \end{cases}
\end{eqnarray*}
so by choosing  $\boldsymbol\varphi=\mathbf{u}\psi_n$ in $\eqref{momweak}$ and $\phi = \theta\psi_n$ in \eqref{entweak} and then summing them up and letting $n\to \infty$, one has
\begin{eqnarray*}
     \int_\Omega \theta(t)+ \frac12 \int_\Omega |\mathbf{u}_t|^2(t) + \frac12 \int_\Omega |\nabla\mathbf{u} |^2(t) \geq \int_\Omega \theta_0+ \frac12 \int_\Omega |\mathbf{v}_0|^2 + \frac12 \int_\Omega |\nabla\mathbf{u}_0 |^2,
\end{eqnarray*}
which compared to $\eqref{enweak}$ implies that $\eqref{enweak}$ has to hold as equality. Let us now prove that $\eqref{entweak}$ holds as identity.
We assume the opposite - there exists a non-negative test function $\tilde\phi$ such that
\begin{eqnarray}
     &&\int_0^T \int_\Omega{\ln\theta} \tilde\phi_t-\int_0^T \int_\Omega \nabla {\ln\theta} \cdot \nabla \tilde\phi\nonumber \\
     &&\quad + \mu\int_0^T \int_\Omega\mathbf{u}_t \cdot \nabla \tilde\phi +\int_0^T \int_\Omega |\nabla{\ln\theta} |^2 \tilde\phi\nonumber \\
     &&< -\int_\Omega {\ln\theta}_0 \tilde\phi_0. \label{strict:ineq}
\end{eqnarray}
and w.l.o.g. $\tilde\phi\leq \theta$. Then, by choosing $\phi=\theta-\tilde\phi$ in \eqref{entweak} and summing it up with \eqref{strict:ineq} gives us
\begin{eqnarray*}
    \int_\Omega \theta(T) + \mu\int_0^T \int_\Omega \theta \nabla \cdot \mathbf{u}_t > \int_\Omega \theta_0,
\end{eqnarray*}
which combined with $\eqref{momweak}$ for $\boldsymbol\varphi=\mathbf{u}$ implies
\begin{eqnarray*}
     &&\int_\Omega \theta(T)+ \frac12 \int_\Omega |\mathbf{u}_t|^2(T) + \frac12 \int_\Omega |\nabla\mathbf{u} |^2(T) \nonumber\\
    &&> \int_\Omega \theta_0+ \frac12 \int_\Omega |\mathbf{v}_0|^2 + \frac12 \int_\Omega |\nabla\mathbf{u}_0 |^2,
\end{eqnarray*}
so we have a contradiction. Thus \eqref{entweak} holds as an identity for all test functions, and therefore pointwise. Now, due to the sufficient regularity of the functions, one has
\begin{eqnarray*}
    \int_0^T\int_\Omega \big({\ln\theta}_t  -\Delta{\ln\theta}  + \mu\nabla\cdot\mathbf{u}_t - |\nabla{\ln\theta}|^2\big)\phi =- \int_0^T\int_{\partial\Omega}\partial_n\ln\theta\phi.
\end{eqnarray*}
For any test function $\varphi\in C^\infty([0,T]\times\partial\Omega)$, we can now choose a sequence of test functions $\phi_n$ converging to $\phi=\varphi\chi_{\partial\Omega}$, so $\int_0^T\int_{\partial\Omega}\partial_n\ln\theta\varphi=0$. Since the test function $\varphi$ was arbitrary, one has $\partial_n {\ln\theta}=0$, so $\partial_n \theta=0$. Finally, by multiplying $\eqref{ent:eq1}$ with $\theta$, one obtains that the energy equation $\eqref{system0}_2$ holds pointwise.

\subsection{Weak-strong uniqueness}\label{WSU:lin}
In this section, we prove the weak-strong uniqueness part of Theorem \ref{main:lin}. To this end, we use a relative entropy method. Our proof mostly follows the ideas from \cite{FNWSUNSF}, where weak-strong uniqueness of Navier-Stokes-Fourier system describing heat-conducting compressible fluids is studied. However, comparing to \cite{FNWSUNSF} (and Section \ref{WSU:nonlin}), how we deal with some terms towards the end of the proof is different. We make use of the specific simple form of the system to avoid dealing with quadratic term $|\theta - \Theta|^2$, since this would require us to have at least quadratic heat energy (see $\eqref{quadtratic:control}$). Last but not least, due to defect measure $\theta$ appearing in the momentum equation $\eqref{momweak}$, the relative entropy method is extended to its measure-valued version.

For a weak solution $(\mathbf{u},\tau)$ in sense of Definition $\ref{weaksolution}$ with measures $\theta,\boldsymbol\sigma$, and smooth functions
$\mathbf{U}\in C_0^\infty([0,T]\times \Omega)$ and $\Theta\in C^\infty([0,T]\times \Omega)$ with $\Theta>0$, $\partial_n \Theta=0$ on $(0,T)\times\partial\Omega$ and $\mathcal{T}:=\ln\Theta$, we define the \textbf{relative entropy inequality}
\begin{eqnarray}
     &&\int_\Omega d\mathcal{E}\left(\theta,\tau\big| \Theta,\mathcal{T} \right)(t)+\frac12\int_\Omega|\mathbf{u}_t - \mathbf{U}_t|^2(t)+\frac12\int_\Omega|\nabla\mathbf{u}  - \nabla\mathbf{U} |^2(t) \nonumber\\
    &&\quad + \int_0^t\int_\Omega|\nabla \tau-\nabla\mathcal{T}|^2\Theta \nonumber\\
     &&\leq \int_\Omega\left(e^{\tau_0}-\Theta(0) -\Theta(0)(\tau_0-\mathcal{T}(0)) \right)\nonumber\\&&+\frac12\int_\Omega|\mathbf{u}_t - \mathbf{U}_t|^2(0)+\frac12\int_\Omega|\nabla\mathbf{u}  - \nabla\mathbf{U} |^2(0)\nonumber\\
    &&\quad+\int_0^t\int_{\Omega} (\mathbf{u}_t-\mathbf{U}_t)\cdot\big(-\mathbf{U}_{tt}+\Delta\mathbf{U} -\mu \nabla\Theta\big)\nonumber\\
    &&\quad-\int_0^t \int_\Omega (\tau-\mathcal{T})( \Theta_t - \Delta \Theta+\mu\Theta\nabla\cdot\mathbf{U}_t) \nonumber\\
    &&\quad - \mu \int_0^t  \int_\Omega\nabla\cdot \mathbf{U}_t~d\mathcal{E}\left(\theta,\tau\big| \Theta,\mathcal{T} \right), \label{relent} 
\end{eqnarray}
where the measure-valued relative entropy is defined as
\begin{eqnarray*}
d\mathcal{E}\left(\theta,\tau\big| \Theta,\mathcal{T} \right):=d\theta - \left(\Theta-\Theta(\tau-\mathcal{T})\right)dx.
\end{eqnarray*}


Moreover, we notice the following:
\begin{lem}\label{wlasnosci}
For $\theta=e^{\tau}dx+f$, where $f\geq 0$ is a singular part of $\theta$, for any $\mathcal{T}=\ln\Theta$,
we have \begin{eqnarray}
    0\leq e^\tau - \Theta-\Theta(\tau-\mathcal{T})\leq   \mathcal{E}\left(\theta,\tau\big| \Theta,\mathcal{T} \right),\label{coercive}
\end{eqnarray}
so $\mathcal{E}\left(\theta,\tau\big| \Theta,\mathcal{T}\right)$ is non-negative and $\mathcal{E}\left(\theta,\tau\big| \Theta,\mathcal{T} \right)=0$ if and only if $\tau=\mathcal{T}$, $e^\tau=\Theta$ and $f=0$.
\end{lem}

\subsubsection{Any weak solution satisfies the relative entropy inequality}

We start with testing $\eqref{momweak}$ with $\chi_{(0,t]}\mathbf{U}_t$ to obtain
\begin{eqnarray*}
    &&\int_{\Omega} (\mathbf{u}_t\cdot \mathbf{U}_t)(t)-\int_0^t\int_{\Omega}\mathbf{u}_t\cdot \mathbf{U}_{tt}+
    \int_0^t\int_{\Omega} \nabla\mathbf{u} :\nabla\mathbf{U}_t - \mu\int_0^t\int_{\Omega} \nabla\cdot\mathbf{U}_td\theta\nonumber \\
   &&=\int_{\Omega}\mathbf{v}_0\cdot \mathbf{V}_0
\end{eqnarray*}
and expressing
\begin{eqnarray*}
    &&\int_0^t\int_{\Omega} \nabla\mathbf{u} :\nabla\mathbf{U}_t = -\int_0^t\int_{\Omega} \mathbf{u}\cdot \Delta\mathbf{U}_t \nonumber \\
   &&= \int_0^t\int_{\Omega} \mathbf{u}_t\cdot \Delta\mathbf{U} - \int_{\Omega} \mathbf{u}\cdot\Delta\mathbf{U} ~\Big|_0^t \\
    &&= \int_0^t\int_{\Omega} \mathbf{u}_t\cdot\Delta\mathbf{U} + \int_{\Omega} \nabla\mathbf{u} :\nabla\mathbf{U}  ~\Big|_0^t,
\end{eqnarray*}
we obtain
\begin{eqnarray}
    &&\int_{\Omega} (\mathbf{u}_t\cdot \mathbf{U}_t)(t) + \int_{\Omega} \big(\nabla\mathbf{u} :\nabla\mathbf{U} \big)(t) \nonumber \\
   &&\quad  +\int_0^t\int_{\Omega} \mathbf{u}_t\cdot\big(-\mathbf{U}_{tt}+\Delta\mathbf{U} \big)
   - \mu\int_0^t\int_{\Omega}\nabla\cdot \mathbf{U}_t d\theta\nonumber\\
   &&=\int_{\Omega}\mathbf{v}_0\cdot \mathbf{V}_0+\int_{\Omega} \nabla\mathbf{u}_0:\nabla\mathbf{U}_0 . \label{id1}
\end{eqnarray}
Since
\begin{eqnarray*}
    &&\int_0^t\int_{\Omega} \mathbf{u}_t\cdot\big(-\mathbf{U}_{tt}+\Delta\mathbf{U} \big)\nonumber \\
   &&=\int_0^t\int_{\Omega} (\mathbf{u}_t-\mathbf{U}_t)\cdot\big(-\mathbf{U}_{tt}+\Delta\mathbf{U} \big)+\int_0^t\int_{\Omega} \mathbf{U}_t\cdot\big(-\mathbf{U}_{tt}+\Delta\mathbf{U} \big)\\
    &&=\int_0^t\int_{\Omega} (\mathbf{u}_t-\mathbf{U}_t)\cdot\big(-\mathbf{U}_{tt}+\Delta\mathbf{U} \big) - \Big( \frac12\int_\Omega|\mathbf{U}_t|^2 +\frac12\int_\Omega |\nabla\mathbf{U} |^2 \Big)\Big|_0^t,
\end{eqnarray*}
we subtract $\eqref{id1}$ from the energy inequality $\eqref{enweak}$ to obtain
\begin{eqnarray}
    &&\int_\Omega\theta(t)+\frac12\int_\Omega|\mathbf{u}_t - \mathbf{U}_t|^2(t)+\frac12\int_\Omega|\nabla\mathbf{u}  - \nabla\mathbf{U} |^2(t)\nonumber\\
   &&\quad -\int_0^t\int_{\Omega} (\mathbf{u}_t-\mathbf{U}_t)\cdot\big(-\mathbf{U}_{tt}+\Delta\mathbf{U} -\mu\nabla\Theta\big) \nonumber \\
   &&\quad -\mu\int_0^t\int_{\Omega} (\mathbf{u}_t-\mathbf{U}_t)\cdot\nabla\Theta + \mu\int_0^t\int_{\Omega} \nabla\cdot \mathbf{U}_td\theta\nonumber\\
   &&\leq \int_\Omega e^{\tau_0}+ \frac12\int_\Omega|\mathbf{v}_0 - \mathbf{V}_0|^2+\frac12\int_\Omega|\nabla\mathbf{u}_0  - \nabla\mathbf{U}_0 |^2, \label{id3}
\end{eqnarray}
where we have also added and subtracted the term $\int_0^t\int_{\Omega} (\mathbf{u}_t-\mathbf{U}_t)\cdot\mu\nabla\Theta$. Next, in $\eqref{entweak}$ we choose $\phi=\chi_{(0,t]}\Theta$ to obtain
\begin{eqnarray}
       &&\int_0^t \int_\Omega{{\tau}} \Theta_t+\int_0^t \int_\Omega {{\tau}}  \Delta \Theta + \mu\int_0^t \int_\Omega\mathbf{u}_t \cdot \nabla \Theta +\int_0^t \int_\Omega |\nabla{{\tau}} |^2\Theta\nonumber \\
       && \leq \int_\Omega ({{\tau}} \Theta)(t)-\int_\Omega {{\tau}}_0 \Theta_0,     \label{id2}
\end{eqnarray}
so by adding and subtracting the terms $\int_0^t \int_{\Omega}{\mathcal{T}} \Theta_t$, $\int_0^t \int_{\Omega}{\mathcal{T}} \Delta \Theta$ and $\mu \int_0^t \int_\Omega \Theta\nabla\cdot\mathbf{U}_t$, one has
\begin{eqnarray}
      &&\int_0^t \int_\Omega({{\tau}}-{\mathcal{T}}) \Theta_t+\int_0^t \int_\Omega ({{\tau}} -{\mathcal{T}}) \Delta \Theta + \int_0^t \int_\Omega{\mathcal{T}} \Delta \Theta  +\int_0^t \int_\Omega |\nabla{{\tau}} |^2\Theta \nonumber\\
      &&\quad+\mu\int_0^t \int_\Omega(\mathbf{u}_t-\mathbf{U}_t) \cdot \nabla \Theta - \mu \int_0^t \int_\Omega \Theta\nabla\cdot\mathbf{U}_t\nonumber\\
      &&\leq \int_\Omega ({{\tau}} \Theta)(t)-\int_\Omega {{\tau}}_0 \Theta_0 -\int_0^t\int_\Omega {\mathcal{T}}\Theta_t. \label{id5}
\end{eqnarray}
Now, since
\begin{eqnarray*}
    &&{\mathcal{T}} \Theta_t = -\Theta_t +({\mathcal{T}}\Theta)_t,\\
    &&\int_0^t \int_{\Omega}{\mathcal{T}} \Delta \Theta= \int_0^t \int_{\Omega}\Delta{\mathcal{T}}  \Theta = \int_0^t \underbrace{\int_{\Omega} \Delta \Theta}_{=0} - \int_0^t \int_{\Omega} |\nabla {\mathcal{T}}|^2\Theta,
\end{eqnarray*}
and
\begin{eqnarray*}
    &&\int_0^t \int_\Omega\big(|\nabla{{\tau}} |^2 -|\nabla{\mathcal{T}}|^2\big)\Theta \nonumber \\
   &&= \int_0^t \int_\Omega|\nabla{{\tau}}  -\nabla{\mathcal{T}}|^2 \Theta + 2\int_0^t \int_\Omega\nabla {\mathcal{T}}\cdot(\nabla {{\tau}} - \nabla{\mathcal{T}})\Theta\\
    &&=\int_0^t \int_\Omega|\nabla{{\tau}}  -\nabla{\mathcal{T}}|^2 \Theta- 2\int_0^t \int_\Omega \nabla \cdot (\underbrace{\nabla{\mathcal{T}}\Theta}_{=\nabla\Theta})( {{\tau}} - {\mathcal{T}})\\
    &&=\int_0^t \int_\Omega|\nabla{{\tau}}  -\nabla{\mathcal{T}}|^2 \Theta - 2\int_0^t \int_\Omega \Delta\Theta( {{\tau}} - {\mathcal{T}}),
\end{eqnarray*}
where we have used the identity  $a^2-b^2 = (a-b)^2 +2b(a-b)$, one arrives at
\begin{eqnarray}
      &&\int_0^t \int_\Omega({{\tau}}-{\mathcal{T}}) \Theta_t+\int_0^t \int_\Omega ({{\tau}}  -{\mathcal{T}})\Delta \Theta +\mu\int_0^t \int_\Omega(\mathbf{u}_t-\mathbf{U}_t) \cdot \nabla \Theta \nonumber \\
      &&- \mu \int_0^t \int_\Omega \Theta\nabla\cdot\mathbf{U}_t +\int_0^t \int_\Omega (|\nabla{{\tau}} |^2-|\nabla{\mathcal{T}}|^2)\Theta\nonumber\\
      &&\leq \int_\Omega ({{\tau}} \Theta-{\mathcal{T}} \Theta+\Theta)(t)-\int_\Omega ({{\tau}}_0 \Theta_0-{\mathcal{T}}_0 \Theta_0+\Theta_0).  \label{id4}
\end{eqnarray}
We now sum up $\eqref{id3}$ and $\eqref{id4}$ to obtain
\begin{eqnarray*}
     &&\int_\Omega d\mathcal{E}\left(\theta,\tau\big| \Theta,\mathcal{T} \right)(t)+\frac12\int_\Omega|\mathbf{u}_t - \mathbf{U}_t|^2(t)+\frac12\int_\Omega|\nabla\mathbf{u}  - \nabla\mathbf{U} |^2(t)\nonumber \\
     && \quad + \int_0^t\int_\Omega|\nabla {{\tau}}-\nabla{\mathcal{T}}|^2\Theta \nonumber\\
     &&\leq \int_\Omega e^{\tau_0}-\Theta(0)-\Theta(0)(\tau_0-\mathcal{T}(0))\nonumber\\
     &&+\frac12\int_\Omega|\mathbf{u}_t - \mathbf{U}_t|^2(0)+\frac12\int_\Omega|\nabla\mathbf{u}  - \nabla\mathbf{u} |^2(0)\nonumber\\
    &&\quad+\int_0^t\int_{\Omega} (\mathbf{u}_t-\mathbf{U}_t)\cdot\big(-\mathbf{U}_{tt}+\nabla\cdot\nabla\mathbf{u} -\mu \nabla\Theta\big)\\
    &&\quad-\int_0^t \int_\Omega({{\tau}}-{\mathcal{T}}) \Theta_t+\int_0^t \int_\Omega ({{\tau}}  -{\mathcal{T}})\Delta \Theta- \mu\int_0^t\int_{\Omega}(d\theta-\Theta) \nabla\cdot \mathbf{U}_t.
\end{eqnarray*}
Finally, noticing that the last three terms can be transformed as follows,
\begin{eqnarray*}
    &&\int_0^t \int_\Omega({{\tau}}-{\mathcal{T}}) \Theta_t-\int_0^t \int_\Omega ({{\tau}}  -{\mathcal{T}})\Delta \Theta+ \mu\int_0^t\int_{\Omega}(d\theta-\Theta) \nabla\cdot \mathbf{U}_t\nonumber
 \\
    &&=\int_0^t \int_\Omega({{\tau}}-{\mathcal{T}})(\Theta_t - \Delta \Theta)\pm\mu\int_0^t \int_\Omega({{\tau}}-{\mathcal{T}}) \Theta\nabla\cdot\mathbf{U}_t \\
    &&\quad +\mu\int_0^t \int_\Omega(d\theta-\Theta) \nabla\cdot \mathbf{U}_t\\
    &&=\int_0^t \int_\Omega (\tau-\mathcal{T})( \Theta_t - \Delta \Theta+\mu\Theta\nabla\cdot\mathbf{U}_t) \\
    &&\quad+ \mu \int_0^t  \int_\Omega\nabla\cdot \mathbf{U}_t d\mathcal{E}\left(\theta,\tau\big| \Theta,\mathcal{T} \right),
\end{eqnarray*}
we arrive at $\eqref{relent}$.

\subsubsection{The proof of weak-strong uniqueness}\label{endgame:lin}
Provided that $(\mathbf{U},e^{\mathcal{T}}=\Theta)$ is a classical solution of $\eqref{system0}$ such that $\Theta_0=\theta_0, \mathbf{U}_0 = \vu_0,$ and $\mathbf{V}_0 = \mathbf{v}_0$, the inequality $\eqref{relent}$ gives us
\begin{eqnarray*}
    &&\int_\Omega d\mathcal{E}\left(\theta,\tau\big| \Theta,\mathcal{T} \right)(t)+\frac12\int_\Omega|\mathbf{u}_t - \mathbf{U}_t|^2(t)+\frac12\int_\Omega|\nabla\mathbf{u}  - \nabla\mathbf{U} |^2(t) \nonumber \\
   &&\quad + \int_0^t\int_\Omega|\nabla {\tau}-\nabla\mathcal{T}|^2\Theta \nonumber\\
    &&\leq-\mu\int_0^t \int_\Omega \nabla\cdot \mathbf{U}_t d\mathcal{E}\left(\theta,\tau\big| \Theta,\mathcal{T} \right)\\
    &&\leq \mu \int_0^t \left(|| \nabla\cdot\mathbf{U}_t||_{L^\infty(\Omega)} \int_\Omega d\mathcal{E}\left(\theta,\tau\big| \Theta,\mathcal{T} \right) \right) \nonumber \\
   &&\leq C \int_0^t \int_\Omega d\mathcal{E}\left(\theta,\tau\big| \Theta,\mathcal{T} \right)(s)~ds,
\end{eqnarray*}
and since $\mathcal{E}\left(\theta,\tau\big| \Theta,\mathcal{T} \right)$ is non-negative and $C>0$ doesn't depend on $t$, we conclude by Gronwall's inequality that $\mathcal{E}\left(\theta,\tau\big| \Theta,\mathcal{T} \right)=0$ for all $t\in[0,T]$. In turn, on the one hand, in view of the above inequality, we have $\mathbf{u}\equiv \mathbf{U}$, on the other hand, Lemma \ref{wlasnosci} yields
$\tau \equiv \mathcal{T},\theta \equiv \Theta$. The proof is finished.

\section{Proof of Theorem \ref{main:nonlin}}
\subsection{Global existence -- the construction of approximate solutions}
\subsubsection{The approximate problem}
Before we start, let us point out that in order to solve $\eqref{APP:nonlin}_2$ given below, one needs at least $C^{2,\alpha}$ boundary, which is higher than Lipschitz. However, this issue can be dealt with by constructing a sequence of smooth domains which converge to $\Omega$, and one can easily show that the weak solutions in the sense of Definition \ref{weaksolution:nonlin} are sequentially stable with respect to this domain convergence. Such construction of domains is originally due to Ne\v{c}as \cite{necas} (in Russian), and it can also be found in \cite[Chapter 1, Theorem 8.3.1]{chipot}. Thus, for simplicity, let $\Omega$ be of $C^{2,\alpha}$ regularity.\\

Now, similarly as in previous section, we introduce a basis $\{\boldsymbol\varphi_i\}_{i\in \mathbb{N}}$ of $H_0^1(\Omega)$ and denote $V_n:={\rm span}\{\boldsymbol\varphi_i\}_{1\leq i\leq n}$. The problem of interest is the following:

\begin{mydef}[Approximate problem]\label{DefApp:nonlin}
We say that $\mathbf{u}\in C^2(0,T; V_n)$ and $\theta\in L^2(0,T; H^2(\Omega))\cap H^1(0,T;L^2(\Omega))$ is a solution to the approximate problem if the following equations are satisfied
\begin{align}
\begin{cases}
     \int_\Omega \mathbf{u}_{tt} \cdot \boldsymbol\varphi+\int_\Omega \nabla\mathbf{u} : \nabla \boldsymbol\varphi -\mu\int_\Omega\theta \nabla\cdot\boldsymbol\varphi=0,~ \text{ for all } t\in(0,T), ~ \boldsymbol\varphi\in V_n,\\[2mm]
    e_1(\theta)_t - \nabla\cdot(\kappa(\theta) \nabla\theta) + \mu \theta \nabla\cdot \mathbf{u}_t+\delta\theta^2-\delta\frac1{\theta^2} = 0,
    \end{cases}\label{APP:nonlin}
\end{align}
where $e_1(\theta)=\theta+\theta^\alpha$ and $\kappa(\theta)=1+\theta^\beta$ and $\delta>0$ is an approximation parameter. The initial data $\mathbf{u}(0,\cdot)=\mathbf{u}_0^n$, $\mathbf{u}_t(0,\cdot)=\mathbf{v}_0^n$ and $\theta(0,\cdot)=\theta_0^n$ are chosen in the same way as in Definition $\ref{DefApp}$.
\end{mydef}

\begin{rem}
(1) The role of the additional terms $\delta\theta^2$ and $-\delta\theta^{-2}$ is to ensure that the approximate temperature is uniformly bounded from below and above by a constant. This is in contrast with the case studied in previous section, since here we are relying on the comparison principle instead of the maximum principle. \\
(2) In this section, the constructed approximate solutions depend on $n, \delta$, so using the notational convention from previous section, we should write $(\mathbf{u}^{n,\delta},\theta^{n,\delta})$. However, in this section we will omit this superscript for simplicity of notation. We will emphasize these indices when passing to the limit. \\
(3) The proof of existence and convergence of approximate solutions presented in this section is inspired by ideas from \cite[Chapter 3]{FN}.
\end{rem}

\subsubsection{Estimates of approximate solutions}
We start with the following:

\begin{lem}\label{EnergyApprox:nonlin}
Let $(\vu,\theta)$ be a solution in the sense of Definiton $\ref{DefApp:nonlin}$ such that $\theta\geq c>0$. Then, $(\mathbf{u},\theta)$ satisfy the entropy balance equation
\begin{eqnarray}
     &&\left(\ln\theta+\frac{\alpha}{\alpha-1}\theta^{\alpha-1}+\mu\nabla\cdot \mathbf{u}\right)_t -\nabla \cdot \left(\frac{\kappa(\theta)\nabla\theta}{\theta} \right)+\delta\theta- \delta\theta^{-3} = \frac{\kappa(\theta)|\nabla\theta|^2}{\theta^2}. \nonumber \\
     && \label{app:ent:bal}
\end{eqnarray}
Moreover, for all $t\in (0,T]$, the following total energy balance
\begin{eqnarray}
    &&\int_\Omega( \theta+\theta^\alpha)(t)+ \frac12 \int_\Omega |\mathbf{u}_t|^2(t) + \frac12 \int_\Omega |\nabla\mathbf{u} |^2(t)  +\delta \int_0^t \int_\Omega\theta^2 \nonumber \\
    &&=\delta \int_0^t \int_\Omega\theta^{-2}+ \int_\Omega (\theta_0^n+(\theta_0^n)^\alpha)+ \frac12 \int_\Omega |\mathbf{v}_0^n|^2 + \frac12 \int_\Omega |\nabla\mathbf{u}_0^n |^2,\label{app:en:id:nonlin}
\end{eqnarray}
and the total dissipation balance
\begin{eqnarray}
     &&\int_\Omega \left(\theta+\theta^{\alpha} -\ln\theta-\frac{\alpha}{\alpha-1}\theta^{\alpha-1}\right) +  \frac12 \int_\Omega |\mathbf{u}_t|^2(t) +\frac12 \int_\Omega |\nabla\mathbf{u} |^2(t)\nonumber\\
     &&\quad + \int_0^t\int_\Omega \frac{(1+\theta^\beta)|\nabla\theta|^2}{\theta^2}+\delta \int_0^t \int_\Omega(\theta^2+\theta^{-3}) \nonumber\\
     &&=\delta \int_0^t \int_\Omega(\theta+\theta^{-2})+ \int_\Omega \left(\theta_0^n+(\theta_0^n)^{\alpha} -\ln\theta_0^n-\frac{\alpha}{\alpha-1}(\theta_0^n)^{\alpha-1}\right)\nonumber \\
     &&\quad + \frac12 \int_\Omega |\mathbf{v}_0^n|^2+ \frac12 \int_\Omega |\nabla\mathbf{u}_0^n|^2, \label{app:diss:nonlin}
\end{eqnarray}
are satisfied.
\end{lem}
\begin{proof}
The entropy balance $\eqref{app:ent:bal}$ is obtained by dividing $\eqref{APP:nonlin}_2$ by $\theta$. The identity $\eqref{app:en:id:nonlin}$ follows by taking $\vphi=\vu_t \chi_{[0,t]}$ in $\eqref{APP:nonlin}_1$ and summing it with $\eqref{APP:nonlin}_2$ integrated over $(0,t)\times \Omega$, while $\eqref{diss:nonlin}$ is obtained by subtracting $\eqref{app:ent:bal}$ integrated over $(0,t)\times \Omega$ from $\eqref{app:en:id:nonlin}$.
\end{proof}

\begin{lem}[Comparison principle]\label{non:comp:pr}
Let $\mathbf{u}\in C^2(0,T; V_n)$ and $\underline{\theta},\overline{\theta}$ be a sub and super solutions to the equation $\eqref{APP:nonlin}_2$, that is they satisfy $\eqref{APP:nonlin}_2$ where $"="$ is replaced with $"\leq"$ and $"\geq"$, respectively. Denote $K(\theta):=\int_0^\theta \kappa(s)~ds=\theta+\frac{\theta^{\beta+1}}{\beta+1}$. If 
\begin{enumerate}
    \item $\underline{\theta},\overline{\theta} \in L^2(0,T; H^1(\Omega))$, $\underline{\theta}_t,\overline{\theta}_t \in L^2((0,T)\times \Omega)$, $\Delta K(\underline{\theta}),\Delta K(\overline{\theta}) \in L^2((0,T)\times \Omega)$;
    \item $0<{\rm ess~inf}_{(0,T)\times \Omega} \underline{\theta}\leq {\rm ess~sup}_{(0,T)\times \Omega} \underline{\theta}< \infty$, \\ $0<{\rm ess~inf}_{(0,T)\times \Omega} \overline{\theta}\leq {\rm ess~sup}_{(0,T)\times \Omega} \overline{\theta}< \infty$;
    \item $\underline{\theta}(0,\cdot)\leq \overline{\theta}(0,\cdot)$,
\end{enumerate}
then, one has
\begin{eqnarray}
    \underline{\theta}\leq \overline{\theta}, \quad \text{a.e. on } (0,T)\times \Omega.
\end{eqnarray}
    
\end{lem}
\begin{proof}
First, it is easy to see that the following inequality holds
\begin{eqnarray}
     &&(e_1(\underline{\theta})-e_1(\overline{\theta}))_t - \Delta (K(\underline{\theta})-K(\overline{\theta}))\nonumber\\
     &&\leq  - \mu (\underline{\theta}-\overline{\theta}) \nabla\cdot \mathbf{u}_t-\delta(\underline{\theta}^2-\overline{\theta}^2)+\delta( \underline{\theta}^{-2}-\overline{\theta}^{-2}). \label{comp:ineq}
\end{eqnarray}
Now, introducing a function
\begin{eqnarray*}
    {\rm sgn}^+(a):=\begin{cases}
        0,& \quad a\leq 0,\\
        1,& \quad a>0,
    \end{cases}
\end{eqnarray*}
we multiply $\eqref{comp:ineq}$ with $ {\rm sgn}^+(e_1(\underline{\theta})-e_1(\overline{\theta}))$ to obtain
\begin{eqnarray*}
    && {\rm sgn}^+(e_1(\underline{\theta})-e_1(\overline{\theta})) \Big[(e_1(\underline{\theta})-e_1(\overline{\theta}))_t -\Delta (K(\underline{\theta})-K(\overline{\theta})) \Big] \\
    &&\leq C {\rm sgn}^+(e_1(\underline{\theta})-e_1(\overline{\theta}))  |\underline{\theta}-\overline{\theta}|,
\end{eqnarray*}
where $C$ depends on $\delta$, $\bb{u}$ and upper and lower bounds for $\underline{\theta}$ and $\overline{\theta}$. Now, since $e_1, K$ and $x\mapsto x$ are all increasing, one has
\begin{eqnarray*}
     {\rm sgn}^+(e_1(\underline{\theta})-e_1(\overline{\theta})) = {\rm sgn}^+(K(\underline{\theta})-K(\overline{\theta})) = {\rm sgn}^+(\underline{\theta}-\overline{\theta}),
\end{eqnarray*}
so the above inequality can be written as
\begin{eqnarray}
    && {\rm sgn}^+(e_1(\underline{\theta})-e_1(\overline{\theta})) \big[e_1(\underline{\theta})-e_1(\overline{\theta})\big]_t
    \nonumber\\
    &&\quad - {\rm sgn}^+(K(\underline{\theta})-K(\overline{\theta}))\Delta \big[ K(\underline{\theta})-K(\overline{\theta})\big]\nonumber \\
    &&\leq C{\rm sgn}^+(\underline{\theta}-\overline{\theta}) |\underline{\theta}-\overline{\theta}|. \label{comp:ineq2}
\end{eqnarray}
Denoting $a^+:=\max\{a,0\}$, and noticing that for any function $f$ that $\partial_t f^+ = {\rm sgn}^+(f) f_t$ and
\begin{eqnarray*}
    \int_\Omega \Delta f~ {\rm sgn}^+(f) \leq 0, \quad \text{if } \partial_n f = 0 \text{ on } \partial\Omega ,
\end{eqnarray*}
provided the mentioned derivatives of $f$ are integrable, we integrate $\eqref{comp:ineq2}$ over $(0,t)\times \Omega$ for any $t\in (0,T)$ to obtain
\begin{eqnarray*}
    \int_\Omega (e_1(\underline{\theta})-e_1(\overline{\theta}))^+(t) \leq C \int_0^t \int_\Omega {\rm sgn}^+(\underline{\theta}-\overline{\theta})|\underline{\theta}-\overline{\theta}| = C \int_0^t \int_\Omega (\underline{\theta}-\overline{\theta})^+.
\end{eqnarray*}
Now, since $x\mapsto e_1(x)$ is strictly increasing and differentiable with derivatives strictly bounded from bellow for all $x>0$, one has that $|\underline{\theta}-\overline{\theta}| \leq c| e_1(\underline{\theta})-e_1(\overline{\theta})|$, where $c$ depends on upper and lower bounds for $\underline{\theta}$ and $\overline{\theta}$ and $e_1$, so it leads to
\begin{eqnarray*}
     \int_\Omega (e_1(\underline{\theta})-e_1(\overline{\theta}))^+(t) \leq C \int_0^t \int_\Omega(e_1(\underline{\theta})-e_1(\overline{\theta}))^+,
\end{eqnarray*}
and the conclusion follows by Gronwall's lemma.

\end{proof}

\begin{lem}\label{non:lin:est:lem}
Assume that $\mathbf{u}\in C^2(0,T; V_n)$ and $\theta\in L^2(0,T; H^2(\Omega))\cap H^1(0,T;L^2(\Omega))$ solve the equation $\eqref{APP:nonlin}_2$. Then, there are constants $0<\underline{\theta}\leq \overline{\theta}<\infty$ such that 
\begin{eqnarray}
    \underline{\theta}\leq  \theta\leq \overline{\theta},\quad \text{a.e. on } (0,T)\times \Omega, \label{up:low:bound}
\end{eqnarray}
and the following estimate holds
\begin{eqnarray}
    &&{\rm ess~sup}_{(0,T)}|| \theta||_{H^1(\Omega)}^2+\int_0^T\left(|| \theta_t||_{L^2(\Omega)}^2 + \left|\left| \Delta K(\theta) \right|\right|_{L^2(\Omega)}^2 \right)\nonumber\\
    &&\leq   C\left(\max\limits_{x\in\Omega}e_1(\theta_0^n ),~ ||K(\theta_0^n )||_{L^2(\Omega)},~ ||\mathbf{u}||_{C^1(0,T;L^2(\Omega))}, \underline{\theta}, ~ \overline{\theta} \right). \label{overall:bound}
\end{eqnarray}
\end{lem}
\begin{proof}
First, it is easy to see that there exists a small enough constant $\underline{\theta}>0$ which satisfies the following inequality
\begin{eqnarray*}
    \mu \underline{\theta} \nabla\cdot \mathbf{u}_t+\delta\underline{\theta}^2 \leq  \delta\frac1{\underline{\theta}^2}
\end{eqnarray*}
so $\underline{\theta}$ is a sub solution to $\eqref{APP:nonlin}_2$. Similarly, we can choose a large enough constant $\overline{\theta}$ so that
\begin{eqnarray*}
   \delta\overline{\theta}^2 \geq  \delta\frac1{\overline{\theta}^2}-  \mu \overline{\theta} \nabla\cdot \mathbf{u}_t,
\end{eqnarray*}
so $\overline{\theta}$ is a super solution to $\eqref{APP:nonlin}_2$. Thus, by previous lemma, the inequality $\eqref{up:low:bound}$ follows.

Next, in order to prove $\eqref{overall:bound}$, first we multiply $\eqref{APP:nonlin}_2$ with $K(\theta)_t$ and integrate over $(0,t)\times \Omega$ which gives us
\begin{eqnarray*}
    &&\int_0^t\int_\Omega \left(1 + \alpha\theta^{\alpha-1}\right)\kappa(\theta)|\theta_t|^2 + \int_\Omega |\nabla K(\theta)|^2(t) \\
    &&=\int_\Omega |\nabla K(\theta_0^n )|^2 - \mu \int_0^t \int_\Omega K(\theta)_t\theta \nabla\cdot \mathbf{u}_t - \delta\int_0^t\int_\Omega (\theta^{2}- \theta^{-2})K(\theta)_t \\
    &&\leq C\left(\max\limits_{x\in\Omega}e_1(\theta_0^n ), ||K(\theta_0^n )||_{L^2(\Omega)}, ||\mathbf{u}||_{C^1(0,T;L^2(\Omega))} \right)+ \frac12 \int_0^t\int_\Omega|\theta_t|^2,
\end{eqnarray*}
since $\theta$ is bounded from above by $\eqref{up:low:bound}$, which yields
\begin{eqnarray*}
    &&\int_0^T\int_\Omega |\theta_t|^2+ {\rm ess~sup}_{(0,T)}||\nabla\theta||_{L^2(\Omega)}^2 \\
    &&\leq  C\left(\max\limits_{x\in\Omega}e_1(\theta_0^n ),~ ||K(\theta_0^n )||_{L^2(\Omega)},~ ||\mathbf{u}||_{C^1(0,T;L^2(\Omega))} \right).
\end{eqnarray*}
Similarly, multiplying with $\Delta K(\theta)$ and using the obtained estimate yields
\begin{eqnarray*}
   &&\int_0^T\int_\Omega |\Delta K(\theta)|^2 +{\rm ess~sup}_{(0,T)}||\nabla\theta||_{L^2(\Omega)}^2\\
    &&\leq C\left(\max\limits_{x\in\Omega}e_1(\theta_0^n ),~ ||K(\theta_0^n )||_{L^2(\Omega)},~ ||\mathbf{u}||_{C^1(0,T;L^2(\Omega))} \right),
\end{eqnarray*}
so the proof is finished.
\end{proof}
\subsubsection{Solving the approximate problem}
\begin{lem}
Let $\mathbf{u}\in C^2(0,T; V_n)$. Then, there exists a solution $\theta\in L^2(0,T; H^2(\Omega))\cap H^1(0,T;L^2(\Omega))$ to the equation $\eqref{APP:nonlin}_2$. Moreover, this solution is unique and satisfies $0<\underline{\theta}\leq \theta \leq \overline{\theta}<\infty$ a.e. on $(0,T)\times \Omega$, where $\underline{\theta},\overline{\theta} $ are constants.
\end{lem}
\begin{proof}
Before we start, in order to write $\eqref{APP:nonlin}_2$ as a quasilinear parabolic PDE, we divide $\eqref{APP:nonlin}_2$ by $1+\alpha\theta^{\alpha+1}$ to obtain
\begin{eqnarray*}
    \theta_t - \frac{(1+\theta^\beta)}{1+\alpha \theta^{\alpha-1}}\Delta \theta +\frac{ \mu \theta \nabla\cdot\mathbf{u}_t   -\beta |\nabla\theta|^2 \theta^{\beta-1}+ \delta\theta^2-\delta\theta^{-2}}{1+\alpha \theta^{\alpha-1}}=0.
\end{eqnarray*}
Now, for any $\omega>0$, in order to ensure uniform parabolicity (at least for a fixed $\omega$), we introduce
\begin{eqnarray*}
    f_\omega(x):=\frac{\sqrt{x^2+\omega^2}}{1+\omega\sqrt{x^2+\omega^2}},
\end{eqnarray*}
and use it to approximate some functions in the equation. This will be denoted with the notation $[\cdot]_\omega:=f_\omega(\cdot)$ (for example $[\theta]_\omega=f_\omega(\theta)$). Note that $\frac{\omega}{1+\omega^2}\leq f_\omega(x)<\frac1\omega$ and
\begin{eqnarray*}
    -1<f_\omega'(x)=\frac{x}{\sqrt{x^2+\omega^2}(1+\omega\sqrt{x^2+\omega^2})}<1,
\end{eqnarray*}
for all $x\in \mathbb{R}$. We aim to solve the following system in the classical sense:
\begin{eqnarray}\label{app:system}
\begin{cases}
    &\theta_t - a(\theta)\Delta \theta  +b(t,x,\theta,\nabla\theta)=0,\\
    &a(\theta)\partial_n\theta = 0, \\
    &\theta(0,\cdot)=\theta_{0,\omega}^n,
    \end{cases}
\end{eqnarray}
where $\theta_{0,\omega}^n$ is a regularization of $\theta_{0}^n$ and
\begin{eqnarray*}
   &&a(\theta):= \frac{[1+\langle \theta \rangle_\omega^\beta]_\omega}{[1+\alpha \langle \theta \rangle_\omega^{\alpha-1}]_\omega},\\[3mm]
   &&b(t,x,\theta,\nabla\theta):=\\
   &&\quad \frac{\mu \theta \nabla\cdot\mathbf{u}_t(t,x) -\beta |\nabla\theta|^2 f_\omega'(1+\langle \theta \rangle_\omega^\beta) (\langle \theta \rangle_\omega^\beta)'+\delta\theta^2-\delta\frac1{\theta^2+\omega^2}}{[1+\alpha \langle \theta \rangle_\omega^{\alpha-1}]_\omega},
\end{eqnarray*}
with $\langle x \rangle_\omega^s$ being the regularization of $x\mapsto |x|^s$, for $s>0$. Since the functions $a$ and $b$ satisfy all the necessary properties, the classical solution to this problem follows by the classical quasilinear parabolic theory (see \cite[Chapter VI, Theorem 7.1]{lady} or \cite[Theorem 10.24]{FN}). Now, by multiplying $\eqref{app:system}_1$ with $[1+\alpha \langle \theta \rangle_\omega^{\alpha-1}]_\omega$, one has
\begin{eqnarray*}
    e_\omega(\theta)_t - \Delta K_\omega(\theta)+\mu \theta \nabla\cdot\mathbf{u}_t+\delta\theta^2-\delta\frac{1}{\theta^2 + \omega^2}=0,
\end{eqnarray*}
where
\begin{eqnarray*}
    e_\omega(x) = \int_0^x [1+\alpha \langle x\rangle_\omega^{\alpha-1}]_\omega \ dx, \quad K_\omega(x) = \int_0^x [1+\langle x\rangle_\omega^\beta]_\omega \ dx.
\end{eqnarray*}
This equation has the same structure as the original one $\eqref{APP:nonlin}_2$ and allows us to prove the comparison principle and the estimates in the same way as in lemmas $\ref{non:comp:pr}$ and $\ref{non:lin:est:lem}$. Passing to the limit in $\omega\to 0$, we obtain a solution in $L^2(0,T; H^2(\Omega))\cap H^1(0,T;L^2(\Omega))$. The uniqueness and the lower and upper bounds now follow by the comparison principle given in Lemma $\ref{non:comp:pr}$.
\end{proof}

\begin{lem}
There exists a solution $\mathbf{u}\in C^2(0,T; V_n)$ and $\theta\in L^2(0,T; H^2(\Omega))\cap H^1(0,T;L^2(\Omega))$  to the approximate problem $\eqref{APP:nonlin}$.
\end{lem}
\begin{proof}
The proof is based on a fixed-point argument and it doesn't differ from the one given in Proposition $\ref{ExistenceApp}$, so we omit it.
\end{proof}

\subsection{Global existence -- the convergence of approximate solutions}
Denote the approximate solution obtained in previous subsection as $(\mathbf{u}^n,\theta^n)$. The goal is to pass to the limit $n\to \infty$ and $\delta \to 0$ in approximate momentum equation $\eqref{APP:nonlin}_1$, approximate entropy balance equation $\eqref{app:ent:bal}$ and the total energy balance $\eqref{app:en:id:nonlin}$, and show that the limiting functions are a weak solution in the sense of Definition $\ref{weaksolution:nonlin}$. This will be done in one limit by fixing  $\delta=1/n$, say.

Before we proceed, let us summarize the uniform estimates for $(\mathbf{u}^n,\theta^n)$, which come from Lemma $\ref{EnergyApprox:nonlin}$ (note that the $\delta$ terms appearing on right-hand sides are absorbed):
\begin{eqnarray}
    &&||\mathbf{u}^n||_{L^\infty(0,T; H^1(\Omega))} + ||\mathbf{u}_t^n||_{L^\infty(0,T; L^2(\Omega))}\leq C,\label{est1:nonlin} \\
    && {\rm ess~sup}_{(0,T)} \int_\Omega (e_1(\theta^n)+ |s_1(\theta^n)|)+ \int_0^T\int_\Omega\frac{\kappa(\theta^n)|\nabla\theta^n|^2}{(\theta^n)^2} \leq C, \label{est2:nonlin}\\
    &&\delta\int_0^T\int_\Omega \left(\frac1{(\theta^n)^3}+(\theta^n)^2 \right) \leq C,\label{est3:nonlin}
\end{eqnarray}
where $e_1(\theta^n)=\theta^n+(\theta^n)^\alpha$,  $s_1(\theta^n)=\ln\theta^n+\frac{\alpha}{\alpha-1}(\theta^n)^{\alpha-1}$ and $\kappa(\theta^n)=1+(\theta^n)^\beta$. Note that this implies $(\theta^n)^{\beta-2}|\nabla\theta^n|^2 \in L^1((0,T)\times \Omega)$, so $\nabla ((\theta^n)^{\frac{\beta}{2}}) \in L^2((0,T)\times \Omega)$ and consequently
\begin{eqnarray}
    ||(\theta^n)^\beta||_{L^1(0,T; L^{3}(\Omega))}\leq C \label{est4:nonlin}.
\end{eqnarray}
As a direct consequence of the above inequalities, one has:
\begin{lem}\label{WeakConvergenceLemma:nonlin}
There exists $\theta\in L^{\infty}(0,T;L^\alpha(\Omega))$, $\vu\in L^{\infty}(0,T;H^1(\Omega))\cap W^{1,\infty}(0,T;L^2(\Omega))$ and $\boldsymbol\sigma \in \mathcal{M}^+([0,T]\times \overline{\Omega})$ such that the following convergence properties hold (on a subsequence)
\begin{eqnarray*}
    &\theta^n \rightharpoonup \theta,& \quad \text{ weakly* in } L^{\infty}(0,T;L^\alpha(\Omega)),
    \\
    &\vu^n \rightharpoonup \vu,& \quad \text{ weakly* in } L^{\infty}(0,T;H^1(\Omega)),
    \\
    &\vu_t^n \rightharpoonup \vu_t,& \quad \text{ weakly* in } L^{\infty}(0,T;L^2(\Omega)),\\
    & \frac{k(\theta^n)|\nabla\theta^n|^2}{(\theta^n)^2}+\delta \frac1{(\theta^n)^3} \rightharpoonup \boldsymbol\sigma, &\quad{\rm weakly*}\;{\rm in}\;\mathcal{M}^+([0,T]\times\overline{\Omega}),
\end{eqnarray*}
where $ \frac{k(\theta)|\nabla\theta|^2}{\theta^2} \leq \boldsymbol\sigma$, in the sense of measures.
\end{lem}
Note that these weak convergences are enough to pass to the limit in the approximate momentum equation $\eqref{APP:nonlin}_1$, so it remains to pass to the limit in the approximate entropy balance equation $\eqref{app:ent:bal}$ and the total energy balance $\eqref{app:en:id:nonlin}$. This is the focus of the remainder of this section.

\subsubsection{Pointwise convergence of $\theta$ -- case $\alpha \neq 2$}
First, note that when $\alpha>3\beta$, then $(\theta^n)^\alpha$ is only in $L^1(\Omega)$ for a.a. $t\in(0,T)$. 
Therefore, to keep the proof general, we choose
\begin{eqnarray*}
    \mathbf{U}_n:=\left[{\ln\theta}^n+\frac{\alpha}{\alpha-1}(\theta^n)^{\alpha-1},~\mu\mathbf{u}_t^n- \frac{\kappa(\theta^n)\nabla\theta^n}{\theta^n}\right],\quad \mathbf{V}_n:=[T_k(\theta^n),0],
\end{eqnarray*}
where $T_k$ is the cut-off function
\begin{eqnarray*}
       T_k(x)=\begin{cases}
       x,& \text{for } x< k,\\
       k,& \text{for } x\geq k.
       \end{cases}
\end{eqnarray*}
We will rely on the following well-known div-curl lemma due to Murat \cite{murat} (in French), stated here for the convenience of the reader (see also \cite[Theorem 10.21]{FN} for proof in English):
\begin{lem} Assume that
\begin{eqnarray*}
    &&\mathbf{U}_n \rightharpoonup \mathbf{U}, \quad \text{weakly in } L^p((0,T)\times \Omega),\\
    &&\mathbf{V}_n \rightharpoonup \mathbf{V}, \quad \text{weakly in } L^q((0,T)\times \Omega),
\end{eqnarray*}
where $1/p+1/q=1/r<1$. In addition\footnote{Here, ${\rm div}_{t,x}  \mathbf{U}_n := \partial_t \mathbf{U}_n^1 + \sum_{i=1}^3\partial_{x_i}\mathbf{U}_n^{i+1}$ is the time-space divergence operator, while ${\rm curl}_{t,x} \mathbf{V}_n:= \nabla_{t,x} \mathbf{V}_n - \nabla_{t,x}^T \mathbf{V}_n$ is the time-space curl operator.}, let $ {\rm div}_{t,x} \mathbf{U}_n$ and ${\rm curl}_{t,x} \mathbf{V}_n$ be precompact in $W^{-1,s}((0,T)\times \Omega)$, for some $s>1$. Then,
\begin{eqnarray*}
   \mathbf{U}_n \cdot \mathbf{V}_n  \rightharpoonup \mathbf{U} \cdot \mathbf{V}, \quad \text{weakly in } L^r((0,T)\times \Omega).
\end{eqnarray*}
\end{lem}
By using this lemma for $p= \min\{ \frac{\alpha}{\alpha-1},2\}$ and any $q>\max\{\alpha,2\}$ and $s\in [1, \frac43)$, one has
\begin{eqnarray}\label{div:curl:id}
       \overline{\left(\ln\theta+\frac{\alpha-1}{\alpha}\theta^{\alpha-1}\right)T_k(\theta)} = \left(\overline{\ln\theta}+\overline{\frac{\alpha-1}{\alpha}\theta^{\alpha-1}}\right)\overline{T_k(\theta)},
\end{eqnarray}
where the notation $\overline{f}$ is used to represent the weak limit of $f^n$. Now, since $\ln x$, $x^{\alpha-1}$ and $T_k(x)$ are non-decreasing functions, one has by \cite[Theorem 10.19(i)]{FN}
\begin{eqnarray*}
       \overline{\ln\theta} ~\overline{T_k(\theta)} \leq \overline{\ln\theta T_k(\theta)}, \qquad \overline{\theta^{\alpha-1}}~\overline{T_k(\theta)}\leq \overline{\theta^{\alpha-1}T_k(\theta)},
\end{eqnarray*}
which together with $\eqref{div:curl:id}$ imply
\begin{eqnarray*}
       \overline{\ln\theta} ~\overline{T_k(\theta)} = \overline{\ln\theta T_k(\theta)}, \qquad \overline{\theta^{\alpha-1}}~\overline{T_k(\theta)}= \overline{\theta^{\alpha-1}T_k(\theta)}.
\end{eqnarray*}
Now, by the second identity and \cite[Theorem 10.19(ii)]{FN}, one concludes\footnote{For the purpose of this theorem, one can extend $x\mapsto x^{\alpha-1}$ with $-|x|^{\alpha-1}$ to $\mathbb{R}^-$, but since $\theta>0$, this makes no difference.}
\begin{eqnarray}\label{weak:lim}
       T_k((\overline{\theta^{\alpha-1}})^{\frac1{\alpha-1}}) = \overline{T_k(\theta)}.
\end{eqnarray}
We want to pass to the limit $k\to \infty$. First, note that by \cite[Corollary 10.2]{FN}
\begin{eqnarray*}
       &&\int_0^T \int_\Omega |T_k(\theta^n) - \theta^n|^s = \int_{\{\theta^n\geq k\}} |k - \theta^n|^s \leq \int_{\{\theta^n\geq k\}} | \theta^n|^s \\
       &&\leq \frac{1}{k^{\alpha-s}} \sup_{n>0} || \theta^n||_{L^\alpha((0,T)\times\Omega )}^\alpha \leq \frac{C}{k^{\alpha-s}},
\end{eqnarray*}
for any $s\in[0,\alpha]$, so
\begin{eqnarray*}
       &&\left|\int_0^T \int_\Omega (\overline{T_k(\theta)}-\theta) \varphi\right| = \lim\limits_{n\to \infty} \left|\int_0^T \int_\Omega (T_k(\theta^n)-\theta^n) \varphi\right| \\
       &&\leq \lim\limits_{n\to \infty} || T_k(\theta^n) - \theta^n||_{L^s(0,T; L^s(\Omega))} ||\varphi||_{L^{s^*}(0,T; L^{s^*}(\Omega))} \\
       &&\leq \frac{C}{k^{\alpha-s}}||\varphi||_{L^{s^*}(0,T; L^{s^*}(\Omega))}.
\end{eqnarray*}
As $k\to \infty$, one concludes that $\overline{T_k(\theta^n)} \rightharpoonup \theta$ in $L^s(0,T; L^s(\Omega))$ for any $s\in [1,\alpha)$, so by passing to the limit $k\to \infty$ in $\eqref{weak:lim}$ implies
\begin{eqnarray*}
       (\overline{\theta^{\alpha-1}})^{\frac1{\alpha-1}} = \theta \implies \overline{\theta^{\alpha-1}} = \theta^{\alpha-1} \text{ a.e. on } (0,T)\times \Omega.
\end{eqnarray*}
Now, if $\alpha>2$, function $x\mapsto x^{\alpha-1}$ for $x\geq 0$ is strictly convex so one can use \cite[Theorem 10.20]{FN} to conclude that $\theta^n \to \theta$ a.e. on $(0,T)\times \Omega$. If $\alpha \in (1,2)$, the same conclusion follows by convexity of $x\mapsto -x^{\alpha-1}$ for $x\geq 0$.

\subsubsection{Pointwise convergence of $\theta$ -- case $\alpha = 2$}
In this case, function $x\mapsto x^{\alpha-1}$ is a linear function which is not strictly convex or concave. Thus, we choose
\begin{eqnarray*}
    \mathbf{U}_n:=\left[{\ln\theta}_n+2\theta^n,\mu\mathbf{u}_t^n- \frac{\kappa(\theta^n)\nabla\theta^n}{\theta^n}\right],\quad \mathbf{V}_n:=[(\theta^n)^2,0],
\end{eqnarray*}
so we can conclude in the same way that
\begin{eqnarray*}
       \overline{\ln\theta}~\overline{\theta^2} = \overline{\ln\theta \theta^2}, \qquad \theta\overline{\theta^2}= \overline{\theta^3}.
\end{eqnarray*}
Now, due to \cite[Theorem 10.19(ii)]{FN}, one has $\overline{\theta^2} = \theta^2$ so the convexity of $x\mapsto x^2$ gives the a.e. convergence of $\theta$ due to \cite[Theorem 10.20]{FN}.
\subsubsection{Passing to the limit}
Here, we make use of a.e. convergence of $\theta^n$ and the uniform bounds to prove the convergence of all the nonlinear terms. First, due to uniform boundedness of $\theta^n$ in $L^\infty(0,T;L^\alpha(\Omega))$ in $\eqref{est2:nonlin}$, one has
\begin{eqnarray*}
    (\theta^n)^{\alpha-1} \to \theta^{\alpha-1},\quad \text{in } L^1((0,T)\times \Omega).
\end{eqnarray*}
Moreover, from $\eqref{est2:nonlin}$ and $\eqref{est3:nonlin}$, one has that $\ln\theta^n \in L^\infty(0,T; L^1(\Omega))$ and $\nabla\ln\theta\in L^2((0,T)\times \Omega)$, so by imbedding $\ln\theta^n\in L^2(0,T;L^6(\Omega))$ which gives us
\begin{eqnarray*}
    \ln\theta^n \to \ln\theta,\quad  \text{in } L^1((0,T)\times \Omega).
\end{eqnarray*}
The above two convergences give us
\begin{eqnarray*}
    s_1(\theta^n) \to s_1(\theta),\quad  \text{in } L^1((0,T)\times \Omega).
\end{eqnarray*}
Next, since $\nabla\ln\theta^n$ and $\nabla((\theta^n)^{\frac\beta2})$ are uniformly bounded in $L^2((0,T)\times \Omega)$, one has
\begin{eqnarray*}
    \nabla\ln\theta^n \rightharpoonup \nabla\ln\theta \text{ and }\nabla((\theta^n)^{\frac\beta2})\rightharpoonup \nabla(\theta^{\frac\beta2}), \quad \text{in } L^2((0,T)\times \Omega).
\end{eqnarray*}
On the other hand, due to $\eqref{est2:nonlin}$ and $\eqref{est4:nonlin}$, by interpolation, one can conclude that $(\theta^n)^{\frac\beta2}\in L^p((0,T)\times \Omega)$, for some $p>2$, so
\begin{eqnarray*}
    (\theta^n)^{\frac\beta2} \to \theta^{\frac\beta2},\quad \text{in }  L^2((0,T)\times \Omega).
\end{eqnarray*}
Combining the previous two convergences, one has
\begin{eqnarray*}
    \nabla\ln\theta^n+(\theta^n)^{\frac\beta2}\nabla((\theta^n)^{\frac\beta2})= \frac{\kappa(\theta^n)\nabla\theta^n}{\theta^n}  \rightharpoonup  \frac{\kappa(\theta)\nabla\theta}{\theta},\quad \text{in } L^1((0,T)\times \Omega).
\end{eqnarray*}
Finally, due to $\eqref{est3:nonlin}$,
\begin{eqnarray*}
    && 0\leq \int_0^T \int_\Omega \delta\theta \leq C\delta^{\frac12}  \left(\int_0^T \int_\Omega \left|\delta^{\frac12} \theta^n \right|^2\right)^{\frac12} \leq C\delta^{\frac12}, \\
    && 0\leq \int_0^T \int_\Omega \delta\frac1{(\theta^n)^2} \leq C\delta^{\frac13} \left(\int_0^T \int_\Omega \left|\delta^{\frac23}\frac1{(\theta^n)^3}\right|^{\frac32} \right)^{\frac23} \leq C \delta^{\frac13},
\end{eqnarray*}
so $\delta\theta^n$ and $\delta\frac1{(\theta^n)^2}$ vanish as $n\to\infty$. Thus, we can pass to the limit in the approximate entropy balance equation $\eqref{app:ent:bal}$ and the total energy balance $\eqref{app:en:id:nonlin}$ and conclude the desired result. Note that the $\delta\int_0^T\int_\Omega (\theta^n)^2$ is positive, so it doesn't affect the limiting energy inequality and therefore doesn't need to vanish.

\subsection{Consistency}
This proof is the same as in the previous section and is therefore omitted.
\subsection{Weak-strong uniqueness}\label{WSU:nonlin}
For the sake of generality of the proof, we will restrain from writing the explicit forms of functions $e_1,s_1$ and $\kappa$ throughout the majority of the proof, and only use their properties when they are needed. As in Section \ref{WSU:lin}, the following proof also follows the ideas from \cite{FNWSUNSF}. \\

Now, let us introduce the \textbf{relative entropy} as
\begin{eqnarray*}
    &&\mathcal{E}(\theta\big|\Theta):= e_1(\theta)-e_1(\Theta)-\Theta (s_1(\theta)-s_1(\Theta))\\
    &&=\theta+\theta^\alpha-\Theta -\Theta^\alpha-\Theta\Big(\ln\theta - \ln\Theta+ \frac{\alpha}{\alpha-1}(\theta^{\alpha-1} -\Theta^{\alpha-1}\big) \Big) ,
\end{eqnarray*}
and for any smooth functions $\mathbf{U}\in C_0^\infty([0,T]\times \Omega)$ and $\Theta\in C^\infty([0,T]\times \Omega)$ with $\Theta>0$, $\partial_n \Theta=0$ on $(0,T)\times\partial\Omega$, the \textbf{relative entropy inequality}
\begin{eqnarray}
    &&\int_\Omega\mathcal{E}(\theta\big|\Theta)(t)+ \frac12\int_\Omega|\mathbf{u}_t - \mathbf{U}_t|^2(t)+\frac12\int_\Omega|\nabla\mathbf{u} -
     \nabla\mathbf{U} |^2(t) \nonumber\\
     &&\quad + \int_0^t \int_\Omega  \frac{\kappa(\theta)|\nabla\theta|^2}{\theta^2}\Theta \nonumber\\     &&\leq\int_\Omega\mathcal{E}(\theta\big|\Theta)(0)+ \frac12\int_\Omega|\mathbf{u}_t - \mathbf{U}_t|^2(0)+\frac12\int_\Omega|\nabla\mathbf{u} -
     \nabla\mathbf{U} |^2(0) \nonumber\\
     &&\quad+\int_0^t\int_{\Omega} (\mathbf{u}_t-\mathbf{U}_t)\cdot\big(-\mathbf{U}_{tt}+\Delta\mathbf{U}  - \mu \nabla \Theta\big)
     -\mu\int_0^t\int_{\Omega}(\theta-\Theta) \nabla\cdot \mathbf{U}_t \nonumber\\
     &&\quad -\int_0^t \int_\Omega\left(s_1(\theta)-s_1(\Theta)\right)\Theta_t +\int_0^t \int_\Omega \frac{\kappa(\theta)}\theta\nabla\theta \cdot \nabla\Theta.  \label{rel:ent:nonlin}
\end{eqnarray}
\subsubsection{Any weak solution satisfies the relative entropy inequality}
First, in the same way as in $\eqref{id3}$, we have
\begin{eqnarray}
    &&\int_\Omega(\theta+\theta^\alpha)(t)+ \frac12\int_\Omega|\mathbf{u}_t - \mathbf{U}_t|^2(t)+\frac12\int_\Omega|\nabla\mathbf{u} -
     \nabla\mathbf{U} |^2(t)\nonumber\\
     &&-\int_0^t\int_{\Omega} (\mathbf{u}_t-\mathbf{U}_t)\cdot\big(-\mathbf{U}_{tt}+\Delta\mathbf{U}  - \mu \nabla \Theta\big)
    \nonumber\\
     &&-\mu\int_0^t\int_{\Omega}(\mathbf{u}_t-\mathbf{U}_t)\cdot\nabla\Theta+ \mu\int_0^t\int_{\Omega}\theta \nabla\cdot \mathbf{U}_t\nonumber\\
   &&\leq \int_\Omega(\theta_0+\theta_0^\alpha)+\frac12\int_\Omega|\mathbf{v}_0 - \mathbf{V}_0|^2+\frac12\int_\Omega|\nabla\mathbf{u}_0  - \nabla\mathbf{U}_0 |^2, \label{id3:nonlin}
\end{eqnarray}
where $\mathbf{v}_0=\mathbf{u}_t(0,\cdot)$ and $\mathbf{V}_0=\mathbf{U}_t(0,\cdot)$. Next, in $\eqref{entweak:nonlin}$ we choose $\phi=\chi_{[0,t]}\Theta$ to obtain
\begin{eqnarray}
       &&\int_0^t \int_\Omega s_1(\theta)\Theta_t-\int_0^t \int_\Omega \frac{\kappa(\theta)}\theta \nabla\theta \cdot \nabla\Theta \nonumber \\
       && \quad + \mu\int_0^t \int_\Omega\mathbf{u}_t \cdot \nabla \Theta +\int_0^t \int_\Omega  \frac{\kappa(\theta) |\nabla\theta|^2 }{\theta^2}\Theta\nonumber\\
       &&\leq \int_\Omega s_1(\theta(t))\Theta(t)- \int_\Omega s_1(\theta_0)\Theta(0), \label{id2:nonlin}
\end{eqnarray}
and since
\begin{eqnarray*}
    &&\int_0^t \int_\Omega s_1(\theta)\Theta_t=\int_0^t \int_\Omega (s_1(\theta) - s_1(\Theta))\Theta_t+\int_0^t \int_\Omega s_1(\Theta)\Theta_t\\
    &&=\int_0^t \int_\Omega (s_1(\theta) - s_1(\Theta))\Theta_t+ \int_0^t\int_\Omega \big((s_1(\Theta)\Theta)_t - e_1(\Theta)_t\big),
\end{eqnarray*}
we conclude that \eqref{rel:ent:nonlin} holds.
\subsubsection{Proof of weak-strong uniqueness}
Now, if $(\tilde{\mathbf{u}},\tilde\theta)$ is a classical solution to the same problem with $\tilde{\mathbf{v}}_0=\mathbf{v}_0$, $\tilde{\mathbf{u}}_0=\mathbf{u}_0$ and $\tilde{\theta}_0=\theta_0$, one has
\begin{eqnarray*}
    &&\int_\Omega\mathcal{E}(\theta\big|\tilde{\theta})(t)+ \frac12\int_\Omega|\mathbf{u}_t - \tilde{\mathbf{u}}_t|^2(t)+\frac12\int_\Omega|\nabla\mathbf{u} -
    \nabla\tilde{\mathbf{u}}|^2(t) \nonumber\\
     &&\quad + \int_0^t \int_\Omega  \frac{\kappa(\theta)|\nabla\theta|^2 }{\theta^2}\tilde\theta\\
     &&\leq   -\mu\int_0^t\int_{\Omega}(\theta-\tilde\theta) \nabla\cdot \tilde{\mathbf{u}}_t-\int_0^t \int_\Omega\left(s_1(\theta)-s_1(\tilde\theta)\right)\tilde{\theta}_t  \\
     &&\quad+\int_0^t \int_\Omega \frac{\kappa(\theta)}\theta \nabla\theta\cdot \nabla\tilde{\theta}.
\end{eqnarray*}
First, we want to transform the RHS. We start with
\begin{eqnarray*}
     &&\int_0^t \int_\Omega\left(s_1(\theta)-s_1(\tilde\theta)\right)\tilde{\theta}_t+\mu\int_0^t\int_{\Omega}(\theta-\tilde\theta) \nabla\cdot \tilde{\mathbf{u}}_t  \nonumber\\
     &&=\int_0^t \int_\Omega\left(s_1(\theta)-s_1(\tilde\theta)_\theta(\theta-\tilde\theta) - s_1(\tilde\theta)\right)\tilde{\theta}_t +\int_0^t \int_\Omega \underbrace{s_1(\tilde\theta)_\theta\tilde\theta_t}_{=s_1(\tilde\theta)_t}(\theta-\tilde\theta) \\
     &&\quad +\mu\int_0^t\int_{\Omega}(\theta-\tilde\theta) \nabla\cdot \tilde{\mathbf{u}}_t \\
     &&=\int_0^t \int_\Omega\left(s_1(\theta)-s_1(\tilde\theta)_\theta(\theta-\tilde\theta) - s_1(\tilde\theta)\right)\tilde{\theta}_t  \\
     &&\quad +  \int_0^t \int_\Omega(\theta-\tilde\theta)(s_1(\tilde\theta)_t+\mu\nabla\cdot \tilde{\mathbf{u}}_t).
\end{eqnarray*}
Moreover, since
\begin{eqnarray*}
    && \int_0^t \int_\Omega(\theta-\tilde\theta)\left(\nabla\cdot\left(\frac{\kappa(\tilde\theta)\nabla\tilde\theta}{\tilde\theta}\right)+\frac{\kappa(\tilde\theta)|\nabla\tilde\theta|^2}{\tilde\theta^2}\right)\nonumber\\
     &&= -\int_0^t \int_\Omega(\nabla\theta-\nabla\tilde\theta)\cdot\nabla\tilde\theta \frac{\kappa(\tilde\theta)}{\tilde\theta}+\int_0^t \int_\Omega(\theta-\tilde\theta)\frac{\kappa(\tilde\theta) \nabla\theta\cdot\nabla\tilde\theta}{\tilde\theta^2},
\end{eqnarray*}
we arrive at
\begin{eqnarray*}
    &&\int_\Omega\mathcal{E}(\theta\big|\tilde{\theta})(t)+ \frac12\int_\Omega|\mathbf{u}_t - \tilde{\mathbf{u}}_t|^2(t)+\frac12\int_\Omega|\nabla\mathbf{u} -
     \nabla\tilde{\mathbf{u}}|^2(t)\\
     &&\quad +\int_0^t \int_\Omega  \frac{\kappa(\theta)|\nabla\theta|^2 }{\theta^2}\tilde\theta-\int_0^t \int_\Omega \frac{\kappa(\theta)}\theta \nabla\theta\cdot \nabla\tilde{\theta} \nonumber\\
     &&\quad-\int_0^t \int_\Omega(\nabla\theta-\nabla\tilde\theta)\cdot \nabla\tilde\theta \frac{\kappa(\tilde\theta)}{\tilde\theta}+\int_0^t \int_\Omega(\theta-\tilde\theta)\frac{\kappa(\tilde\theta)|\nabla\tilde\theta|^2}{\tilde\theta^2}\\
     &&\leq -\int_0^t \int_\Omega\left(s_1(\theta)-s_1(\tilde\theta)_\theta(\theta-\tilde\theta) - s_1(\tilde\theta)\right)\tilde{\theta}_t.
\end{eqnarray*}
At this point, we need some specific form for $\kappa(\theta)$ in order to obtain good estimates with the right signs. Therefore, we choose $\kappa(\theta)=1+\theta^2$. Now, we calculate the terms corresponding to the constant part of $\kappa$
\begin{eqnarray*}
    &&\int_0^t \int_\Omega  |\nabla\ln\theta|^2\tilde\theta-\int_0^t \int_\Omega \frac{\nabla\theta\cdot\nabla\tilde\theta}\theta \nonumber\\
     &&\quad -\int_0^t \int_\Omega(\nabla\theta-\nabla\tilde\theta)\cdot\nabla\ln\tilde\theta+\int_0^t \int_\Omega(\theta-\tilde\theta)|\nabla\ln\tilde\theta|^2\\
    &&= \int_0^t \int_\Omega  |\nabla\ln\theta|^2\tilde\theta-\int_0^t \int_\Omega \tilde\theta\nabla\ln\theta\cdot\nabla\ln\tilde\theta  \nonumber\\
     &&\quad -\int_0^t \int_\Omega(\theta\nabla\ln\theta-\tilde\theta\nabla\ln\tilde\theta)\cdot\nabla\ln\tilde\theta+\int_0^t \int_\Omega(\theta-\tilde\theta)|\nabla\ln\tilde\theta|^2 \\
    &&=\int_0^t \int_\Omega  |\nabla\ln\theta-\nabla\ln\tilde\theta|^2\tilde\theta+\int_0^t \int_\Omega (\tilde\theta-\theta)\nabla\ln\theta\cdot\nabla\ln\tilde\theta \nonumber\\
     &&\quad +\int_0^t \int_\Omega(\theta-\tilde\theta)|\nabla\ln\tilde\theta|^2\\
    &&=\int_0^t \int_\Omega  |\nabla\ln\theta-\nabla\ln\tilde\theta|^2\tilde\theta + \int_0^t \int_\Omega (\tilde\theta-\theta)\nabla\ln\tilde\theta \cdot (\nabla\ln\theta-\nabla\ln\tilde\theta).
\end{eqnarray*}
Similarly, corresponding to the second component $-\theta^2$
\begin{eqnarray*}
    &&\int_0^t \int_\Omega  |\nabla\theta|^2\tilde\theta-\int_0^t \int_\Omega \theta\nabla\theta\cdot\nabla\tilde\theta \nonumber\\
     &&\quad -\int_0^t \int_\Omega(\nabla\theta-\nabla\tilde\theta)\cdot\nabla\tilde\theta+\int_0^t \int_\Omega(\theta-\tilde\theta)|\nabla\tilde\theta|^2\\
    &&=\int_0^t \int_\Omega  |\nabla\theta-\nabla\tilde\theta|^2\tilde\theta + \int_0^t \int_\Omega (\tilde\theta-\theta)\nabla\tilde\theta \cdot (\nabla\theta-\nabla\tilde\theta).
\end{eqnarray*}
Thus, we finally obtain
\begin{eqnarray}
     &&\int_\Omega\mathcal{E}(\theta\big|\tilde{\theta})(t)+ \frac12\int_\Omega|\mathbf{u}_t - \tilde{\mathbf{u}}_t|^2(t)+\frac12\int_\Omega|\nabla\mathbf{u} -
     \nabla\tilde{\mathbf{U}}|^2(t)\nonumber\\
     &&\quad+\int_0^t \int_\Omega  |\nabla\theta-\nabla\tilde\theta|^2\tilde\theta
     +\int_0^t \int_\Omega  |\nabla\ln\theta-\nabla\ln\tilde\theta|^2\tilde\theta
     \nonumber\\
     &&\leq -\int_0^t \int_\Omega\left(s_1(\theta)-s_1(\tilde\theta)_\theta(\theta-\tilde\theta) - s_1(\tilde\theta)\right)\tilde{\theta}_t  \nonumber\\
     &&\quad-\int_0^t \int_\Omega (\tilde\theta-\theta)\nabla\tilde\theta \cdot (\nabla\theta-\nabla\tilde\theta)-\int_0^t \int_\Omega (\tilde\theta-\theta)\nabla\ln\tilde\theta \cdot (\nabla\ln\theta-\nabla\ln\tilde\theta)\nonumber\\
     &&\leq \int_0^t  || \tilde\theta_t||_{L^\infty(\Omega)}\int_\Omega |(s_1(\theta)-s_1(\tilde\theta)_\theta(\theta-\tilde\theta) - s_1(\tilde\theta)| \nonumber\\
     &&\quad+ \int_0^t C(\underline{\theta})\big(|| \nabla\tilde\theta||_{L^\infty(\Omega)}+ || \nabla\ln\tilde\theta||_{L^\infty(\Omega)} \big) \int_\Omega   |\theta-\tilde\theta|^2\nonumber\\
     &&\quad+\frac{\underline{\theta}}2\int_0^t \int_\Omega  |\nabla\theta-\nabla\tilde\theta|^2
     +\frac{\underline{\theta}}2\int_0^t \int_\Omega  |\nabla\ln\theta-\nabla\ln\tilde\theta|^2 .  \label{almost:there}
\end{eqnarray}
Note that it is enough to prove
\begin{eqnarray*}
    &&|(s_1(\theta)-s_1(\tilde\theta)_\theta(\theta-\tilde\theta) - s_1(\tilde\theta)| \leq C\mathcal{E}(\theta\big|\tilde{\theta}) \quad \text{ and } \quad |\theta-\tilde\theta|^2 \leq C\mathcal{E}(\theta\big|\tilde{\theta}),
\end{eqnarray*}
for $C>0$ that doesn't depend on $t$, since this implies
\begin{eqnarray*}
    \mathcal{E}(\theta\big|\tilde{\theta})(t) \leq C\int_0^t \mathcal{E}(\theta\big|\tilde{\theta})(s)~ds,
\end{eqnarray*}
and the conclusion follows as in section \ref{endgame:lin}. The following result will be useful:
\begin{lem}\label{lemma:fg}
Let $f,g:[a,b]\to \mathbb{R}$ be two $C^2$ functions and $s\in[a,b]$. If:
\begin{enumerate}
    \item $f(s)=g(s)=f'(s)=g'(s)=0$,
    \item $f''(s)>0$,
    \item $f(x)>0$, on $[a,b]\setminus\{s\}$,
\end{enumerate}
then there is a constant $c>0$ such that
\begin{eqnarray*}
    f(x)\geq c|g(x)|, \quad\text{on } [a,b].
\end{eqnarray*}
\end{lem}
\begin{proof}
By Taylor's theorem, we have
\begin{eqnarray*}
    &&f(x)=\underbrace{f(s)+f'(s)(x-s)}_{=0}+\frac{f''(s)}2(x-s)^2+r_f(x)(x-s)^2\\
     &&=\left(\frac{f''(s)}2+r_f(x) \right)(x-s)^2, \quad x\in[a,b],
\end{eqnarray*}
where $r_f$ is a continuous function such that $r_f(s)=0$. Now, since $f>0$ on $[a,b]\setminus\{s\}$, one has $\frac{f''(s)}2+r_f(x)>0$ on $[a,b]\setminus\{s\}$, while the continuity of $r_f$ and $r_f(s)=0$ implies that $\frac{f''(s)}2+r_f(x) >0$ on entire $[a,b]$. Denote $m=\min\limits_{x\in[a,b]}\frac{f''(s)}2+r_f(x)$. Applying Taylor's theorem to $g$, in the same way we have
\begin{eqnarray*}
    g(x)=\left(\frac{g''(s)}2+r_g(x) \right)(x-s)^2,\quad x\in[a,b],
\end{eqnarray*}
so denoting $M=\max\limits_{x\in[a,b]}\left|\frac{g''(s)}2+r_g(x)\right|$ and choosing $c=\frac{m}M$, we conclude
\begin{eqnarray*}
    c|g(x)| \leq cM(x-s)^2 = m(x-s)^2 \leq \left(\frac{f''(s)}2+r_f(x) \right)(x-s)^2=f(x),
\end{eqnarray*}
and the proof is finished.
\end{proof}
Now, applying this lemma
\begin{eqnarray*}
    |s_1(\theta)-s_1(\tilde\theta)_\theta(\theta-\tilde\theta)-s_1(\tilde\theta)| \leq C \mathcal{E}(\theta\big|\tilde{\theta}), \quad \text{for } \theta\in \left[\underline{\theta},\overline{\theta}\right].
\end{eqnarray*}
Note that here for each fixed $\tilde{\theta}$ there a constant $c_{\tilde\theta}$ for which this inequality holds, so our $C>0$ is chosen as $\max\limits_{\tilde\theta\in [\underline{\theta},\overline{\theta}]}c_{\tilde\theta}$. Next, we aim to control the function $|s_1(\theta)-s_1(\tilde\theta)_\theta(\theta-\tilde\theta)-s_1(\tilde\theta)|$ outside of $\left[\underline{\theta}, \overline{\theta}\right]$. First,\\
\begin{eqnarray*}
    |s_1(\theta)-s_1(\tilde\theta)_\theta(\theta-\tilde\theta)-s_1(\tilde\theta)| \leq C(1+|\ln(\theta)|)\leq C\mathcal{E}(\theta\big|\tilde{\theta}), \quad\text{for } \theta \in \left[0,\underline{\theta}\right], \\
\end{eqnarray*}
since the term $\ln(\theta)$ has a negative sign in $\mathcal{E}(\theta\big|\tilde{\theta})$ and the other terms are bounded, while for $\theta \geq \overline{\theta}$ the function $\theta^\alpha$ dominates $\theta^{\alpha-1}$ and $\ln\theta$ so\\
\begin{eqnarray*}
       |s_1(\theta)-s_1(\tilde\theta)_\theta(\theta-\tilde\theta)-s_1(\tilde\theta)| \leq C(1+e_1(\theta))\leq  C\mathcal{E}(\theta\big|\tilde{\theta}), \quad \text{for } \theta \geq \overline{\theta}, \\
\end{eqnarray*}
which together imply
\begin{eqnarray*}
     |s_1(\theta)-s_1(\tilde\theta)_\theta(\theta-\tilde\theta)-s_1(\tilde\theta)| \leq C \mathcal{E}(\theta\big|\tilde{\theta}).
\end{eqnarray*}
Next, to control $|\theta-\tilde\theta|^2$, we start with
\begin{eqnarray}\label{quadtratic:control}
    |\theta-\tilde\theta|^2 \leq  C\mathcal{E}(\theta\big|\tilde{\theta}),\quad \text{for } \theta \in [0,\underline{\theta}] \cup [\underline{\theta}, \overline{\theta}],
\end{eqnarray}
by Lemma \ref{lemma:fg}
and
\begin{eqnarray*}
     |\theta-\tilde\theta|^2 \leq C(1+\theta^2) \leq C(1+ e_1(\theta)) \leq C\mathcal{E}(\theta\big|\tilde{\theta}),\quad \text{for } \theta \geq \overline{\theta},
\end{eqnarray*}
where we have used the growth condition $\alpha\geq 2$, so combining the above inequalities implies
\begin{eqnarray*}
     |\theta-\tilde\theta|^2 \leq C\mathcal{E}(\theta\big|\tilde{\theta}).
\end{eqnarray*}
Thus, the proof is finished.

\section*{Appendix}
\begin{lem*}\label{weak:time:reg}
Let $\Omega$ be a Lipschitz domain and let $p\in (1,\infty]$. Assume that $\theta \in L^p(0,T;\mathcal{M}(\Omega))$, and that $\mathbf{u}\in L^p(0,T;H^1(\Omega))$ with $\mathbf{u}_t \in L^\infty(0,T;L^2(\Omega))$ satisfies the following equation
\begin{eqnarray*}
     \int_0^T\int_\Omega \mathbf{u}_t \cdot \boldsymbol\varphi_t - \int_0^T\int_\Omega \mathbb{D}(\mathbf{u}): \nabla \boldsymbol\varphi + \mu\int_0^T{}_{\mathcal{M}(\Omega)}\langle\theta, \nabla\cdot\boldsymbol\varphi\rangle_{C_0(\Omega)} =-\int_\Omega \mathbf{v}_0\cdot \boldsymbol\varphi,
\end{eqnarray*}
for all test function $\boldsymbol\varphi \in H^1(0,T;C^{\infty}_0(\Omega))$, $\vphi(T)=0$, where $\mathbf{v}_0\in L^2(\Omega)$ is given. Then, $\mathbf{u}_t \in C_w(0,T;L^2(\Omega))$.
\end{lem*}

\begin{proof}
First, let us prove that . Since $\theta\in L^p(0,T;\mathcal{M}(\Omega))$, one has from the equation for every $\boldsymbol\varphi\in C_0^\infty((0,T)\times \Omega)$ that
\begin{eqnarray*}
    \left|\int_0^T\int_\Omega\mathbf{u}_{tt}\cdot\boldsymbol\varphi \right|\leq \big(||\mathbf{u}||_{L^p(0,T;H^1(\Omega))}+||\theta||_{L^p(0,T;\mathcal{M}^+(\Omega))}\big)||\boldsymbol\varphi||_{L^{p^*}(0,T; W^{2,\infty}(\Omega))},
\end{eqnarray*}
so $\mathbf{u}_{tt} \in L^{p^*}(0,T; [C_0^\infty(\Omega)]')$, and consequently $\mathbf{u}_t\in C(0,T; [C_0^\infty(\Omega)]')$.

Now, for an arbitrary $\boldsymbol\varphi\in L^{2}(\Omega)$, let $\boldsymbol\varphi_n\in C_0^\infty(\Omega)$ be such that $||\boldsymbol\varphi_n-\boldsymbol\varphi||_{L^{2}(\Omega)}\to 0$, as $n\to \infty$. Then, for almost all $t_1, t_2$, one has
\begin{eqnarray*}
    &&\int_\Omega(\mathbf{u}_t(t_1)-\mathbf{u}_t(t_2))\cdot\boldsymbol\varphi\\
    &&=\int_\Omega(\mathbf{u}_t(t_1)-\mathbf{u}_t(t_2))\cdot(\boldsymbol\varphi-\boldsymbol\varphi_n)+\int_\Omega(\mathbf{u}_t(t_1)-\mathbf{u}_t(t_2))\cdot\boldsymbol\varphi_n,
\end{eqnarray*}
so
\begin{eqnarray*}
     &&\lim\limits_{t_2\to t_1}\left|\int_\Omega(\mathbf{u}_t(t_1)-\mathbf{u}_t(t_2))\cdot\boldsymbol\varphi\right|\nonumber\\
     &&\leq \lim\limits_{t_2\to t_1}\left|\int_\Omega(\mathbf{u}_t(t_1)-\mathbf{u}_t(t_2))\cdot(\boldsymbol\varphi-\boldsymbol\varphi_n)\right|+ \underbrace{\lim\limits_{t_2\to t_1}\left|\int_\Omega(\mathbf{u}_t(t_1)-\mathbf{u}_t(t_2))\cdot\boldsymbol\varphi_n\right|}_{=0}\nonumber\\ 
     &&\leq 2||\mathbf{u}_t||_{L^\infty(0,T;L^2(\Omega))}||\boldsymbol\varphi-\boldsymbol\varphi_n||_{L^2(\Omega)} \to 0, \quad \text{as }n\to\infty. 
\end{eqnarray*}
Now, we can give meaning to $ \int_\Omega\mathbf{u}_t(t)\cdot\boldsymbol\varphi$ for all $t$ by constructing a sequence of $t_n$ such that $\int_\Omega\mathbf{u}_t(t_n)\cdot\boldsymbol\varphi$ is finite for all $n$ and $t_n\to t$ as $n\to\infty$. From the above inequality, $\int_\Omega\mathbf{u}_t(t_n)\cdot\boldsymbol\varphi$ forms a Cauchy sequence which converges to a unique limit $a_{\boldsymbol\varphi}\in \mathbb{R}$ for any $\boldsymbol\varphi$, independent of the choice of the sequence $\{t_n\}_{n\in\mathbb{N}}$. Now, since $\left|\int_\Omega\mathbf{u}_t(t_n)\cdot\boldsymbol\varphi \right|\leq C||\varphi||_{L^2(\Omega)}$, one has that $|a_{\boldsymbol\varphi}|\leq C||\varphi||_{L^2(\Omega)}$. This combined with linearity of $\boldsymbol\varphi \mapsto a_{\boldsymbol\varphi}$ allows us to identify $a_{\boldsymbol\varphi}$ with a function $\mathbf{u}_t(t)\in L^2(\Omega)$, by the Riesz theorem. The proof is now finished.\\
\end{proof}

{\bf Acknowledgements.} We thank an anonymous Referee for several useful suggestions and comments.
\\
T.C. was supported by the National Science Center of Poland grant SONATA BIS 7 number UMO-2017/26/E/ST1/00989. B.M. was partially supported by the Croatian Science Foundation (Hrvatska Zaklada za Znanost) grant number IP-2018-01-3706. S.T. was supported by Provincial Secretariat for Higher Education and Scientific Research of Vojvodina, Serbia, grant no 142-451-2593/2021-01/2.

\end{document}